\documentclass[11pt]{article}
\usepackage{amssymb,amsmath,amsfonts,amsthm}

\usepackage{epsfig,color}

\numberwithin{equation}{section}

\newcommand{\R}{{\mathbb R}}
\newcommand{\e}{\varepsilon}

\newtheorem{theorem}{Theorem}[section]

\newtheorem{corollary}[theorem]{Corollary}

\newtheorem{definition}[theorem]{Definition}
\newtheorem{example}[theorem]{Example}

\newtheorem{lemma}[theorem]{Lemma}

\newtheorem{proposition}[theorem]{Proposition}
\newtheorem{remark}[theorem]{Remark}

\begin{document}

\title{\vskip-0.3in Positive solutions for quasilinear elliptic inequalities \\and systems with nonlocal terms}

\author{Marius Ghergu\footnote{School of Mathematics and Statistics,
    University College Dublin, Belfield, Dublin 4, Ireland; {\tt
      marius.ghergu@ucd.ie}}\;\,\footnote{Institute of Mathematics Simion Stoilow of the Romanian Academy, 21 Calea Grivitei St., 010702 Bucharest, Romania}
      $\;\;$        $\;$
{Paschalis Karageorgis\footnote{School of Mathematics,
    Trinity College Dublin; {\tt
        pete@maths.tcd.ie}}}
 $\;\;$    and    $\;$
{Gurpreet Singh\footnote{School of Mathematics,
    Trinity College Dublin; {\tt
        gurpreet.bajwa2506@gmail.com}}}
}

\maketitle

\begin{abstract}
We investigate the existence and nonexistence of positive solutions for the quasilinear elliptic
inequality $L_\mathcal{A} u= -{\rm div}[\mathcal{A}(x, u, \nabla u)]\geq (I_\alpha\ast u^p)u^q$
in $\Omega$, where $\Omega\subset \R^N, N\geq 1,$ is an open set. Here $I_\alpha$ stands for the Riesz potential of order $\alpha\in (0, N)$,
$p>0$ and $q\in \R$. For a large class of operators $L_\mathcal{A}$ (which includes the
$m$-Laplace and the $m$-mean curvature operator) we obtain optimal ranges of exponents $p,q$ and
$\alpha$ for which positive solutions exist.  Our methods are then extended to quasilinear
elliptic systems of inequalities.
\end{abstract}


\medskip



\section{Introduction}\label{sec1}
In this paper we are concerned with the following quasilinear elliptic inequality
\begin{equation}\label{go}
L_\mathcal{A} u= -{\rm div}[\mathcal{A}(x, u, \nabla u)]\geq (I_\alpha\ast u^p)u^q \quad\mbox{ in }\Omega,
\end{equation}
and its associated systems. Here  $\alpha\in (0, N)$, $p>0$ and $q\in \R$.

We investigate \eqref{go} for three particular classes of open sets $\Omega\subset \R^N$, $N\geq 1$, namely:
\begin{itemize}
\item $\Omega$ is open and bounded;
\item $\Omega$ is the whole space $\R^N$;
\item $\Omega$ is the exterior of a closed ball in $\R^N$.
\end{itemize}
In the last case, the inequality \eqref{go} will be considered in $\Omega=\R^N\setminus
\overline{B}_1$ but the arguments  we construct in the following are valid if $\Omega$ is the
complement of any smooth and nondegenerate compact set.

The quantity $I_\alpha\ast u^p$ represents the convolution operation $$ (I_\alpha\ast
u^p)(x)=\int_\Omega I_\alpha(x-y)u^p(y) dy\,, $$ where $I_\alpha:\R^N\to \R$ is the {\it Riesz
potential} of order $\alpha\in (0,N)$  given by $$
I_\alpha(x)=\frac{A_\alpha}{|x|^{N-\alpha}}\quad \mbox{ with }\;
A_\alpha=\frac{\Gamma\big(\frac{N-\alpha}{2}\big)}{\Gamma(\frac{\alpha}{2}\big) \pi^{N/2}2^\alpha}=
C(N, \alpha)> 0. $$ Throughout this paper, $\mathcal{A}: \Omega \times [0, \infty)\times \R^N \to
\R^N$ is a Caratheodory function, that is, $\mathcal{A}$ is measurable and $\mathcal{A}(x, u, \cdot
): \R^N \to \R^N$ is continuous for all $(x, u)\in \Omega \times [0, \infty)$.

The operator $L_\mathcal{A}$ is assumed to fulfill one of the  structural conditions below.

\begin{definition}\label{defW}
Let $m> 1$.

\begin{itemize}

\item We say that $\mathcal{A}$ is $W$-$m$-$C$ ({\it weakly-m-coercive}) if
\begin{equation*}
\mathcal{A}(x, u, \eta)\cdot\eta \geq C|\mathcal{A}(x, u, \eta)|^{m'}
\quad\mbox{ for all } (x, u, \eta)\in \Omega \times[0, \infty)\times \R^N,
\end{equation*}
where $C>0$ is a constant and $m'$ is the H\"{o}lder conjugate of $m>1$, that is,
$\frac{1}{m}+\frac{1}{m'}= 1$.

\item  We say that $\mathcal{A}$ is $S$-$m$-$C$ ({\it strongly-m-coercive}) if
\begin{equation*}
\mathcal{A}(x, u, \eta) \cdot \eta \geq C_1|\eta|^m \geq C|\mathcal{A}(x, u, \eta)|^{m'}
\quad\mbox{ for all } (x, u, \eta)\in \Omega \times[0, \infty)\times \R^N,
\end{equation*}
where $C,C_1>0$ are constants.

\item We say that $\mathcal{A}$ is of mean curvature type $(H_m)$ if
$$ \mathcal{A}= \mathcal{A}(\eta): \R^N\to \R^N, \;\;\;\; \mathcal{A}_{i}(\eta)= A(|\eta|)\eta_i,
$$ where $A\in C[0, \infty)\cap C^{1}(0, \infty)$, $t\longmapsto tA(t)$ is nondecreasing and
there exists $M>1$ such that
\begin{equation*}
\left\{
\begin{aligned}
A(t)&\leq Mt^{m-2} &&\quad\mbox{ for all }t> 0,\\
A(t)&\geq M^{-1}t^{m-2} &&\quad\mbox{ for all } 0< t< 1.
\end{aligned}
\right.
\end{equation*}

\end{itemize}
\end{definition}

\begin{example}
Let $\mathcal{A}: \Omega \times [0,\infty)\times \R^N \to \R^N$ be defined by $\mathcal{A}_{i}(x,
u, \eta)= \displaystyle\sum_{j=1}^{N} a_{ij}(x, u, \eta)\eta_i$. Then $\mathcal{A}$ is $W$-$m$-$C$
provided that there exists $C>0$ such that we have $$ \displaystyle\sum_{i, j=1}^{N} a_{ij}(x, u,
\eta)\eta_i \eta_j \geq C\Big[\displaystyle\sum_{i=1}^{N}\Big(\displaystyle\sum_{j=1}^{N} a_{ij}(x,
u, \eta)\eta_j \Big)\Big]^{m'/2} \quad\mbox{ for all } (x, u, \eta)\in  \Omega \times
[0,\infty)\times \R^N. $$

\end{example}

\begin{example}
The standard $m$-Laplace operator $\mathcal{A}(x, u, \eta)= |\eta|^{m-2}\eta$ is $(H_m)$ and
$S$-$m$-$C$ for any $m>1$.
\end{example}

\begin{example}
The $m$-mean curvature operator given by
\begin{equation}\label{mco}
\mathcal{A}(x, u, \eta)= \frac{|\eta|^{m-2}}{\sqrt{1+|\eta|^m}}\eta
\end{equation}
is $W$-$m$-$C$ but not $S$-$m$-$C$. Also, $\mathcal{A}(x, u, \eta)$ is $(H_m)$ provided that $m\geq
2$.
\end{example}

\begin{example}
Assume $\mathcal{A}_i(x, u, \eta)= A(x, u, |\eta|)\eta_i$, where $A: \Omega \times [0,
\infty)\times [0, \infty)\to \R$. Then:
\begin{itemize}
\item $\mathcal{A}$ is $W$-$m$-$C$ if there exists $M>1$ such that
\begin{equation*}
\left\{
\begin{aligned}
0\leq A(x, u, t)&\leq Mt^{m-2} &&\quad\mbox{ for all }t> 0,\\
A(x, u, t)&\geq M^{-1}t^{m-2} &&\quad\mbox{ for all } 0< t< 1.
\end{aligned}
\right.
\end{equation*}
\item  $\mathcal{A}$ is $S$-$m$-$C$ if
\begin{equation}\label{smc}
M^{-1}t^{m-2}\leq A(x, u, t)\leq Mt^{m-2}\quad\mbox{  for all } t> 0,
\end{equation}
for some constant $M>1$.
\end{itemize}
\end{example}

\begin{definition}
We say that $u\in C(\Omega)\cup W_{loc}^{1, 1}(\Omega)$ is a positive  solution of \eqref{go} if
\begin{itemize}
\item $u>0$, $\mathcal{A}(x, u, \nabla u)\in L^{1}_{loc}(\Omega)^{N}$,
$L_{\mathcal{A}}u\in L^{1}_{loc}(\Omega)$, $(I_\alpha*u^p)u^q\in L^{1}_{loc}(\Omega)$.
\item
\begin{equation}\label{c1}
\int_{\Omega}\frac{u^{p}(y)}{1+|y|^{N-\alpha}}dy< \infty.
\end{equation}

\item for any $\phi \in C_{c}^{\infty}(\Omega)$, $\phi\geq 0$ we have
\begin{equation*}
\int_{\Omega}\mathcal{A}(x, u, \nabla u)\cdot \nabla \phi \geq \int_{\Omega}(I_\alpha*u^p)u^q \phi.
\end{equation*}
\end{itemize}
\end{definition}

Condition \eqref{c1} follows from the fact that $I_\alpha*u^p< \infty$ in $\Omega$.

\begin{remark}

(i) If $\mathcal{A}$ is $S$-$m$-$C$ then $L_{\mathcal{A}}$ satisfies the weak Harnack inequality
(see \cite[Theorem 2]{T1967}).

(ii) If $\mathcal{A}$ satisfies $(H_m)$ then the standard maximum principle for $L_{\mathcal{A}}$
holds (see \cite[Remark 2.3]{BP2001}). Thus, whenever $\mathcal{A}$ is $(H_m)$ and the exponent $q$
in \eqref{go} is positive, we may relax the positivity condition on $u$ to $u\geq 0$.
\end{remark}

The study of elliptic inequalities in unbounded domains goes back to early 1980s although elliptic
equations in $\R^N$ have been discussed, for radially symmetric solutions, at least one century ago
by Emden \cite{E1907} and Fowler \cite{F1914,F1920}. In the celebrated paper \cite{GS1981}, Gidas
and Spruck obtained that the semilinear equation $-\Delta u=u^p$ in $\R^N$, $N\geq 3$, has no
$C^2$-solutions for $1\leq p<\frac{N+2}{N-2}$ and that the upper exponent $\frac{N+2}{N-2}$ is
sharp. Instead, if one considers the related inequality  $-\Delta u\geq u^p$ in $\R^N$, $N\geq 3$
then the optimal range for nonexistence changes to $1\leq p<\frac{N}{N-2}$ and this new upper
exponent $\frac{N}{N-2}$ is also sharp (see, e.g., \cite{BCN1994,MP2001}). Since then, such results have
been extended to many differential operators. 

For instance, the authors in \cite{CDM2008} (see also \cite{DM2010}) discuss Liouville type results for 
$$
 -{\rm div}[\mathcal{A}(x,u, \nabla u)]\geq f(u) \quad\mbox{ in }\Omega.
 $$
The approach in \cite{CDM2008} relies essentially on representation formulae for linear inequalities, nonlinear capacity methods and the weak form of Harnack's inequality.
In \cite{BM1998} the authors use a blow-up argument to derive a priori bounds for solutions and thus to obtain Liouville  type results for inequalities and their corresponding systems.
Other types of problems may be found in \cite{BFP2015,KLM2005,KLZ2003,
LLM2007, MP2001}. 

A systematic study of the inequality
$$
 L_\mathcal{A} u= -{\rm div}[\mathcal{A}(x,u, \nabla u)]\geq |x|^{\sigma}u^q \quad\mbox{ in }\Omega,
$$
along with the corresponding system $$ \left\{
\begin{aligned}
&L_\mathcal{A} u= -{\rm div}[\mathcal{A}(x, u, \nabla u)]\geq |x|^au^pv^q\\
&L_\mathcal{B} v= -{\rm div}[\mathcal{B}(x, v, \nabla v)]\geq |x|^b u^rv^s
\end{aligned}
\right.\qquad\mbox{ in } \Omega,
$$
is carried out in \cite{BP2001} for various domains $\Omega\subset \R^N$, such as open balls and their
complements, half balls and half spaces.

The quasilinear elliptic inequality 
${\rm div}(A(|\nabla u|)\nabla u )\geq f(u)$
is discussed in \cite{PRS2007, PSZ1999} in connection with the strong maximum principle and the compact support principle. 
More recently, quasilinear elliptic inequalities and systems integrate the gradient term in the nonlinearity: the authors in \cite{FPR2010} discuss the quasilinear coercive inequality
$$
{\rm div}(g(x)|\nabla u|^{p-2}\nabla u)\geq h(x)f(u)\ell(|\nabla u|)\quad\mbox{ in }\R^N.
$$
Systems of quasilinear elliptic inequalities of type
$$
\left\{
\begin{aligned}
& -{\rm div}(h_1(x)A(|\nabla u|)\nabla u)\geq f(x,u,v,\nabla u, \nabla v)\\
& -{\rm div}(h_2(x)B(|\nabla v|)\nabla v)\geq g(x,u,v,\nabla u, \nabla v)
\end{aligned}
\right.\qquad\mbox{ in } \R^N
$$
and
$$
\left\{
\begin{aligned}
& -{\rm div}[\mathcal{A}(x, u, \nabla u)]\geq a(x) u^{p_1}v^{q_1}|\nabla u|^{\theta_1}\\
& -{\rm div}[\mathcal{B}(x, v, \nabla v)]\geq b(x) u^{p_2}v^{q_2}|\nabla u|^{\theta_2}
\end{aligned}
\right.\qquad\mbox{ in } \R^N
$$
are considered in \cite{F2011} and \cite{F2013} respectively.

To the best of our knowledge, the first results dealing with quasilinear elliptic inequalities featuring nonlocal terms appear in \cite{CMP2008}. The authors in \cite{CMP2008} obtain 
local estimates and Liouville type results for 
$$
 -{\rm div}[\mathcal{A}(x,u, \nabla u)]\geq K\ast u^q \quad\mbox{ in }\R^N,
 $$
where ${\mathcal A}$ is $S$-$m$-$C$, $K\in L^{1}_{loc}(\R^N)$, $K\geq 0$ and $q>0$.

Our study of \eqref{go} and its associated systems is motivated by \cite{MV2013} where the authors
consider the semilinear elliptic inequality
\begin{equation}\label{mv}
-\Delta u+\frac{\lambda}{|x|^\gamma}u\geq (I_\alpha\ast u^p)u^q\quad\mbox{ in }\R^N\setminus\overline{B}_1.
\end{equation}
Related inequalities or equations are discussed in \cite{CZ2016,CZ2018,GT2016}. The case of
equality in \eqref{mv} is motivated by the so-called Choquard (or Choquard-Pekar) equation in quantum physics; the reader may consult \cite{MV2017} for a mathematical account on this topic.

In this paper we are concerned with inequality \eqref{go} which is the quasilinear version of
\eqref{mv} in the case $\lambda=0$. The influence of the singular term
$\frac{\lambda}{|x|^\gamma}u$ to \eqref{go} will be discussed in a forthcoming work \cite{GKS}. In
our current study of \eqref{go} we recover the global picture of existence and nonexistence of
solutions obtained in the semilinear case in \cite[Theorem 1]{MV2013}. Specifically, if
$\mathcal{A}$ is $(H_m)$ with $m>1$, we are able to find optimal conditions for $\alpha\in (0,N)$,
$p>0$ and $q\in \R$ such that positive solutions to \eqref{go} exist in exterior domains. One
relevant ingredient in the approach of \eqref{mv} is the nonlocal version of the
Agmon-Allegretto-Piepenbrink positivity principle (see \cite[Proposition 3.2]{MV2013}) whose proof
uses essentially the linear character of the differential operator. This tool does not seem to be
available for \eqref{go} and we shall rely instead on a priori estimates devised in \cite{BP2001}
for quasilinear equations with local terms (see Proposition \ref{p1} below).   Next, we turn to the
study of \eqref{go} in bounded open sets. In such a setting we obtain that if $\mathcal{A}$ is
essentially $(H_m)$, $p>0$ and $p+q\neq m-1$, then \eqref{go} has a positive radial solution.
Finally, we investigate the existence and nonexistence of positive solutions to some systems driven
by \eqref{go}.

The paper is organized as follows. Section 2 contains the main results of this work. In Section 3
we collect some preliminary facts which will be used in our study of \eqref{go}. Sections 4-9
contain the proofs of our main results.

Throughout this paper by $c, C, C_0, C_1, C_2,...$ we denote positive generic constants whose values may vary on each occasion. Also, all integrals are computed in the Riemann sense even if we omit the $dx$ or $dy$ symbol.
\medskip

\section{Main Results}

\subsection{Nonexistence of solutions}
Our first nonexistence result concerns the general case where $\mathcal{A}$ is $W$-$m$-$C$.
\begin{theorem}\label{thm2}
Let $\Omega= \R^N\setminus \overline B_1$. Assume $\mathcal{A}$ is $W$-$m$-$C$, $N> m> 1$, and one
of the following condition holds.
\begin{itemize}
\item[{\rm (i)} ] $p+q> m-1$, $q\leq m-1$ and $q< m-1-\frac{N-\alpha-m}{N}p$.
\item[{\rm (ii)} ] $p+q> m-1$, $q< m-1$ and $q=  m-1-\frac{N-\alpha-m}{N}p$.
\item[{\rm (iii)} ] $p+q\leq m-1$.
\end{itemize}
Then \eqref{go} has no positive solutions.
\end{theorem}

In case $\mathcal{A}$ is $(H_m)$ we may obtain further nonexistence results.
\begin{theorem}\label{thm1}
Let $\Omega= \R^N\setminus \overline B_1$. Assume $\mathcal{A}$ is $(H_m)$ for some $m>1$.
\begin{enumerate}
\item[\rm (i)] If $m\geq N$ then \eqref{go} has no positive solutions.
\item[\rm (ii)] If $N>m> 1$ and one of the following conditions hold:
\begin{enumerate}
\item[\rm (ii1)] $0< p\leq \frac{\alpha(m-1)}{N-m}$;
\item[\rm (ii2)] $m-1< q\leq \frac{\alpha(m-1)}{N-m}$ and $\, \alpha> N-m$;
\item[\rm (ii3)] $m-1\leq p+q\leq \frac{(N+\alpha)(m-1)}{N-m}$;
\end{enumerate}
then \eqref{go} has no positive solutions.
\end{enumerate}
\end{theorem}

\subsection{Existence of solutions in $\mathbb R^N$}
In the following we assume that $\mathcal{A}$ has the form
\begin{equation}\label{AA}
\mathcal A(x,u, \nabla u) = A(|\nabla u|) \nabla u
\end{equation}
for some function $A\in C[0,\infty)\cap C^1(0,\infty)$ that satisfies
\begin{equation} \label{con1}
\text{$A(t) \geq Ct^{m-2}$ for small enough $t>0$,}
\end{equation}
and
\begin{equation} \label{upper}
\frac{t A'(t)}{A(t)} \leq m-2 \quad\text{for small enough $t>0$,}
\end{equation}
where $C>0$ and $m>1$.

Clearly, for any $m>1$ the $m$-Laplace operator and the $m$-mean curvature operator as given in
\eqref{mco} satisfy \eqref{AA}-\eqref{upper}. More generally, if $A(t) = t^{m-2} f(t)$, where $f$
is continuously differentiable at $t=0$ and such that $f(0)>0\geq f'(t)$ for small enough $t>0$
then \eqref{con1}-\eqref{upper} are fulfilled. Indeed, one can easily check that
\begin{equation*}
\frac{t A'(t)}{A(t)} = m-2 + \frac{tf'(t)}{f(t)} \leq m-2 \quad\text{for small enough $t>0$.}
\end{equation*}

\begin{theorem}\label{thm3}
Assume \eqref{AA}-\eqref{upper}, $N>m>1$, and one of the following conditions hold.
\begin{itemize}
\item[\rm (i)]
$p > \frac{N(m-1)}{N-m}$, $q= m-1$ and $\alpha = N-m$.

\item[\rm (ii)]
$p \geq \frac{N(m-1)}{N-m}$, $q> m-1 -p \frac{N-\alpha - m}{N} $.

\end{itemize}
Then, \eqref{go} has a bounded radial solution in $\Omega= \mathbb R^N$.
\end{theorem}

\begin{theorem}\label{thm4}
Assume \eqref{AA}-\eqref{upper} and $N>m>1$.   If
\begin{equation} \label{three}
p> \frac{\alpha(m-1)}{N-m}, \qquad q> \frac{\alpha(m-1)}{N-m} \quad\mbox{ and }\quad
p+q > \frac{(\alpha +N)(m-1)}{N-m},
\end{equation}
then, \eqref{go} has a bounded radial solution in $\Omega= \mathbb R^N$.
\end{theorem}
From the above results we have a complete picture for all exponents $p>0$, $q\in \R$ in case ${\mathcal A}$ is $(H_m)$ and satisfies \eqref{upper}. 
\begin{corollary}\label{corh1}
Assume $\mathcal{A}$ is $(H_m)$ and satisfies \eqref{upper} for some $m>1$ and let $p>0$, $q\in \R$ and $\alpha\in (0,N)$.
\begin{enumerate}
\item[\rm (i)] If $m\geq N$ then \eqref{go} has no positive solutions.
\item[\rm (ii)] Assume $N>m>1$. 
Then, the following statements are equivalent:
\begin{enumerate}
\item[\rm (ii1)] Inequality \eqref{go} has a positive solution in $\R^N$.
\item[\rm (ii2)] Inequality \eqref{go} has a positive solution in some exterior domain.
\item[\rm (ii3)] The following conditions  hold:
\end{enumerate}
$$
\begin{aligned}
&p>\frac{\alpha(m-1)}{N-m}\,, \; p+q> \frac{(N+\alpha)(m-1)}{N-m} \; \mbox{ and }\; q>
m-1-\frac{N-\alpha-m}{N}p &&\mbox{ if } 0<\alpha<N-m,\\
& p>m-1\,,\; q\geq m-1\;\mbox{ and } \; p+q> \frac{(N+\alpha)(m-1)}{N-m} &&\mbox{ if } \alpha=N-m,\\
& \min\{p,q\}>\frac{\alpha(m-1)}{N-m}\;\mbox{ and } \; p+q> \frac{(N+\alpha)(m-1)}{N-m} &&\mbox{ if } N-m<\alpha<N.
\end{aligned}
$$
\end{enumerate}
\end{corollary}
The diagrams below illustrate the cases $0<\alpha<N-m$ and $N-m\leq \alpha<N$ respectively. 

\begin{figure}[ht]
\begin{center}
  \includegraphics[width=6.5in ]{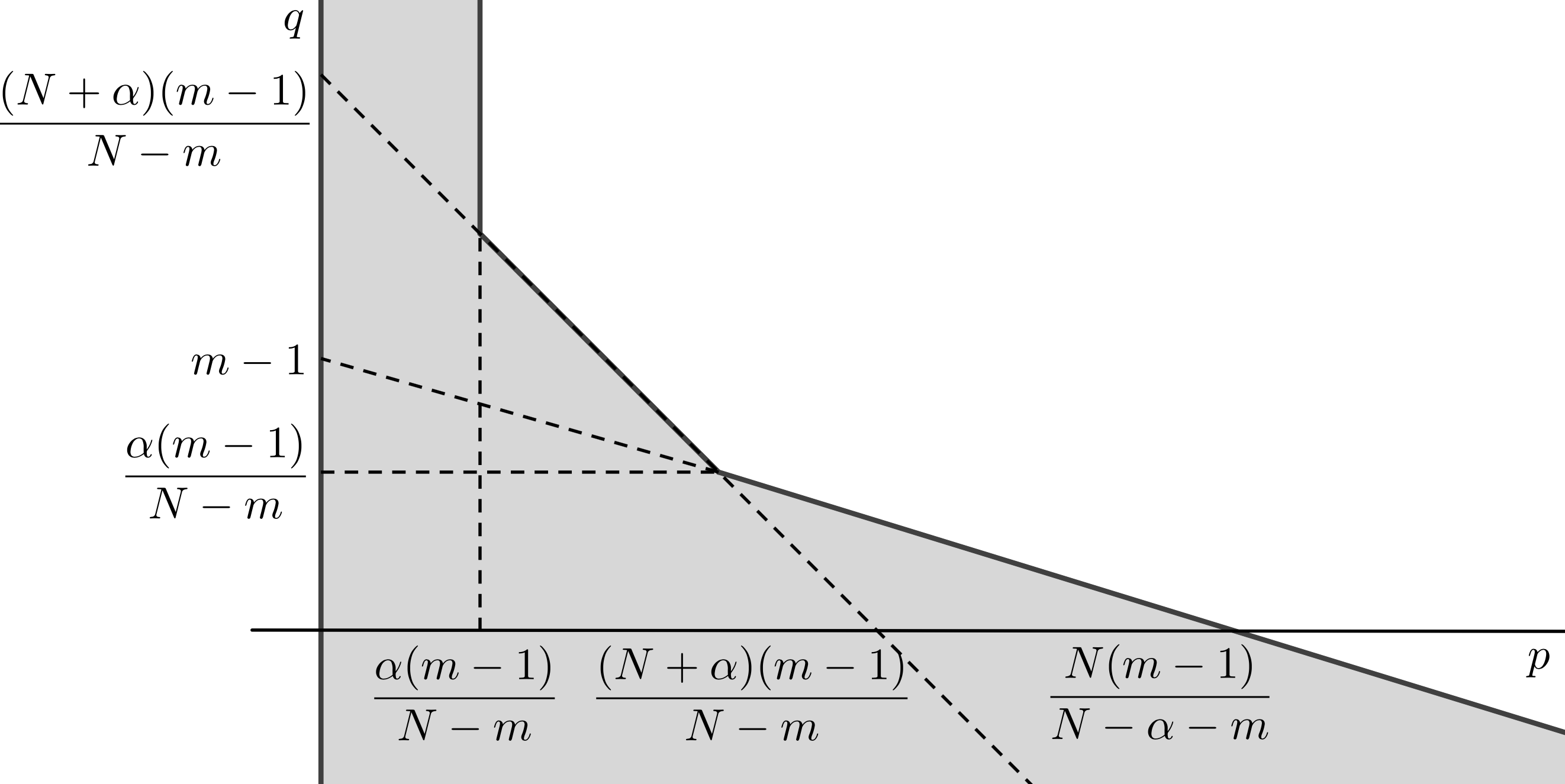}
  \caption{The nonexistence region (shaded) for positive solutions to \eqref{go} in the case $0<\alpha<N-m$.}
  \label{fig:case1}
  \end{center}
\end{figure}

\begin{figure}[ht!]
\begin{center}
  \includegraphics[width=6.2in ]{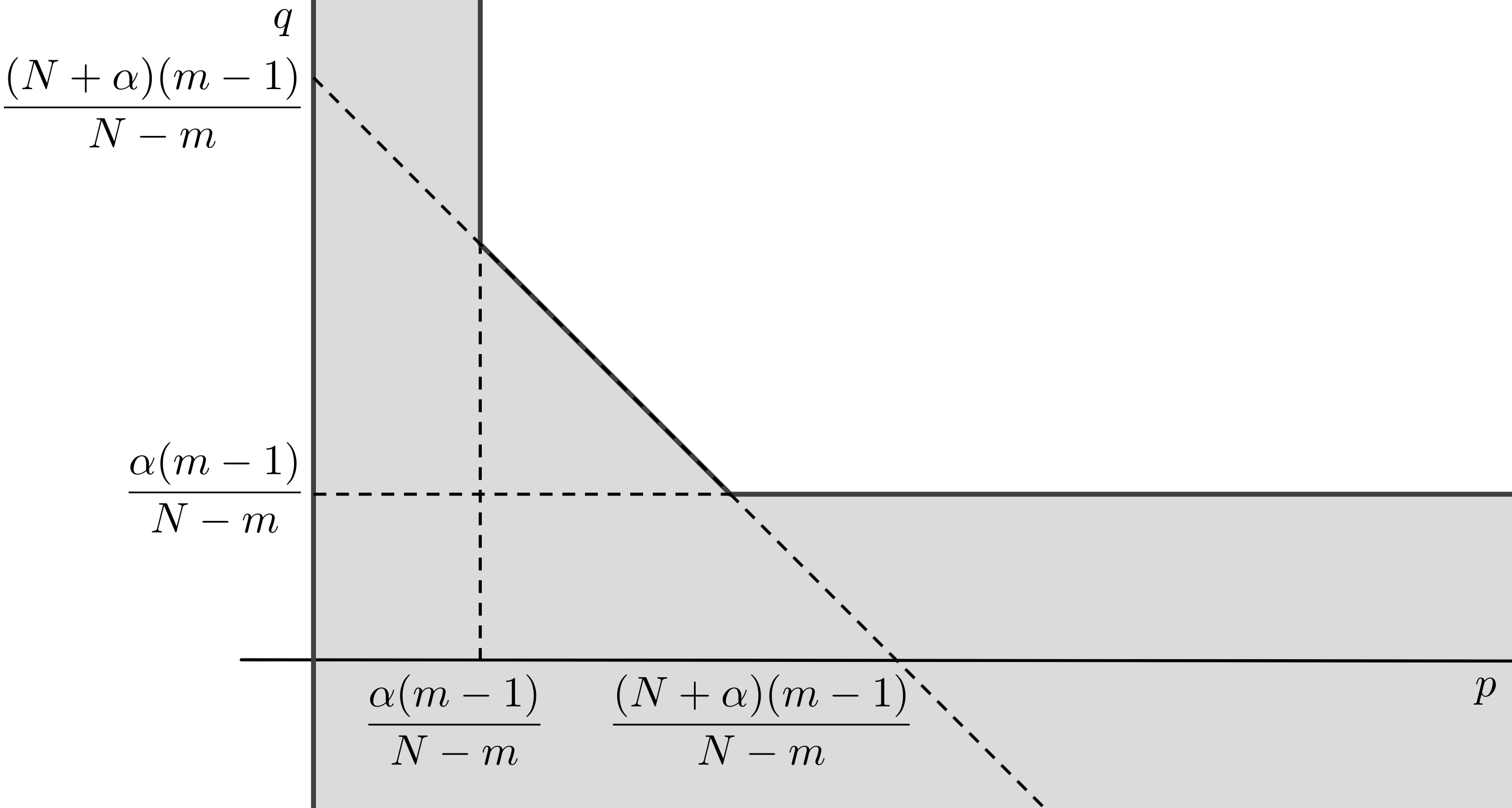}
  \caption{The nonexistence region (shaded) for positive solutions to \eqref{go} in the case $N-m\leq \alpha<N$.}
  \label{fig:case2}
  \end{center}
\end{figure}

\subsection{Existence of solutions in bounded domains}
In this section we assume that $\mathcal{A}$ satisfies \eqref{AA} and study inequality \eqref{go}
in open bounded sets. Our main goal is to show that solutions exist under very mild conditions on
$p,q$.

The first two theorems of this section deal with the case $p+q>m-1$.  For $N>1$ we have:

\begin{theorem} \label{thm5}
Assume $\mathcal{A}$ satisfies \eqref{AA} and \eqref{con1}. If $p>0$, $p+q>m-1$ and $N>1$, then
\eqref{go} has a bounded radial solution in any bounded open set $\Omega \subset \mathbb R^N$.
\end{theorem}
For $N=1$ we have a similar statement as follows.
\begin{theorem} \label{thm6}
Assume that $\mathcal{A}$ satisfies \eqref{AA} and
\begin{equation} \label{con2}
\text{$tA'(t) + A(t) \geq C_0 t^{m-2}$ for small enough $t>0$,}
\end{equation}
where $C_0>0$ and $m>1$.  If $p>0$, $p+q>m-1$ and $N=1$, then \eqref{go} has a bounded radial
solution in any bounded open set $\Omega \subset \mathbb R^N$.
\end{theorem}

Conditions \eqref{AA}, \eqref{con1} and \eqref{con2} are clearly satisfied by the $m$-Laplace and
the $m$-mean curvature operators, for any $m>1$. More generally, these structural conditions are
fulfilled for $A(t) = t^{m-2} f(t)$, where $f$ is continuously differentiable at $t=0$ with
$f(0)>0$.  This is easy to check since
\begin{equation*}
\lim_{t\to 0^+} \frac{tA'(t) + A(t)}{t^{m-2}} = \lim_{t\to 0^+} \, [(m-1)f(t) + tf'(t)] =
(m-1) f(0).
\end{equation*}

Next, we establish two similar results for the case $p+q<m-1$.  In this setting, our previous
approach applies almost verbatim, but the assumptions \eqref{con1}, \eqref{con2} that we imposed
for small enough $t$ must now hold for large enough $t$. This is true for the $m$-Laplace operator
for any $m>1$.

\begin{theorem}\label{thm7}
Assume $\mathcal{A}$ satisfies \eqref{AA} and
\begin{equation} \label{con3}
\text{$A(t) \geq C_0t^{m-2}$ for large enough $t>0$,}
\end{equation}
where $C_0>0$ and $m>1$.  If $p>0$, $p+q<m-1$ and $N>1$, then \eqref{go} has a bounded radial
solution in any bounded open set $\Omega \subset \mathbb R^N$.
\end{theorem}

\begin{theorem}\label{thm8}
Assume $\mathcal{A}$ satisfies \eqref{AA} and
\begin{equation} \label{con4}
\text{$tA'(t) + A(t) \geq C_0 t^{m-2}$ for large enough $t>0$,}
\end{equation}
where $C_0>0$ and $m>1$.  If $p>0$, $p+q<m-1$ and $N=1$, then \eqref{go} has a bounded radial
solution in any bounded interval $\Omega \subset \mathbb R$.
\end{theorem}

\medskip

\subsection{Systems}
In this section we consider three systems driven by the inequality \eqref{go}. More precisely we
investigate

\begin{equation}\label{sys1}
\left\{
\begin{aligned}
&L_\mathcal{A} u= -{\rm div}[\mathcal{A}(x, u, \nabla u)]\geq (I_\alpha*v^p)u^q\\
&L_\mathcal{B} v= -{\rm div}[\mathcal{B}(x, v, \nabla v)]\geq (I_\beta*u^r)v^s
\end{aligned}
\right.
\qquad\mbox{ in } \R^N\setminus \overline{B}_{1},
\end{equation}
\smallskip

\begin{equation}\label{sys2}
\left\{
\begin{aligned}
&L_\mathcal{A} u= -{\rm div}[\mathcal{A}(x, u, \nabla u)]\geq (I_\alpha*v^p)v^q\\
&L_\mathcal{B} v= -{\rm div}[\mathcal{B}(x, v, \nabla v)]\geq (I_\beta*u^r)u^s
\end{aligned}
\right.\qquad\mbox{ in } \R^N\setminus \overline{B}_{1},
\end{equation}
\smallskip

\begin{equation}\label{sys3}
\left\{
\begin{aligned}
&L_\mathcal{A} u= -{\rm div}[\mathcal{A}(x, u, \nabla u)]\geq (I_\alpha*u^p)v^q \\
&L_\mathcal{B} v= -{\rm div}[\mathcal{B}(x, v, \nabla v)]\geq (I_\beta*v^r)u^s
\end{aligned}
\right.
\qquad\mbox{ in } \R^N\setminus \overline{B}_{1},
\end{equation}
where $\alpha, \beta \in (0, N)$, $p,r>0$, $q, s\in \R$ and $\mathcal{A}$, $\mathcal{B}$ are
$W$-$m_1$-$C$ and $W$-$m_2$-$C$ respectively, for some $m_1$, $m_2> 1$. Solutions of
\eqref{sys1}-\eqref{sys3} are understood in the weak sense as we made precise in Definition
\ref{defW} for the single inequality \eqref{go}.

The result below provides sufficient conditions for nonexistence of solutions to
\eqref{sys1}-\eqref{sys3} for $W$-$m$-$C$ operators.

\begin{theorem}\label{thm9}
Assume $\mathcal{A}$ is $W$-$m_1$-$C$ and $\mathcal{B}$ is $W$-$m_2$-$C$ for some $m_1$, $m_2> 1$,
$\alpha, \beta \in (0, N)$, $p,r>0$, $q, s\in \R$.
\begin{enumerate}
\item[\rm (i)] If
\begin{equation}\label{s1}
2N-(m_1+m_2+\alpha+\beta)\leq N\Big(\frac{m_1-1-q}{r}+\frac{m_2-1-s}{p}\Big),
\end{equation}
\begin{equation}\label{s2}
p\geq m_2-1-s\geq 0 \quad\mbox{ and } \quad r\geq m_1-1-q\geq 0,
\end{equation}
then, \eqref{sys1} has no positive solutions.
\item[\rm (ii)] If
\begin{equation}\label{s3}
2N-(m_1+m_2+\alpha+\beta)\leq N\Big(\frac{m_2-1-q}{p}+\frac{m_1-1-r}{s}\Big),
\end{equation}
\begin{equation}\label{s4}
p\geq m_2-1-q\geq 0 \quad \mbox{ and } \quad r\geq m_1-1-s\geq 0,
\end{equation}
then, \eqref{sys2} has no positive solutions.
\item[\rm (iii)] If
\begin{equation}\label{s5}
2N-(m_1+m_2+\alpha+\beta)\leq N\Big(\frac{m_1-1-s}{p}+\frac{m_2-1-q}{r}\Big),
\end{equation}
\begin{equation}\label{s6}
p\geq m_1-1-s\geq 0 \quad \mbox{ and } \quad r\geq m_2-1-q\geq 0,
\end{equation}
then, \eqref{sys3} has no positive solutions.
\end{enumerate}
\end{theorem}

\begin{theorem}\label{thm10}
Assume $\mathcal{A}$ is $(H_{m_1})$ and $\mathcal{B}$ is $(H_{m_2})$ for some $m_1, m_2>1$ and let
$\alpha, \beta \in (0, N)$.
\begin{enumerate}
\item[\rm (i)] If one of the following conditions hold
\begin{enumerate}
\item[\rm (i1)] $\alpha>N-m_1$ and $m_1-1<q\leq \frac{\alpha(m_1-1)}{N-m_1}$;

\item[\rm (i2)] $\beta>N-m_2$ and $m_2-1<s\leq \frac{\beta(m_2-1)}{N-m_2}$;

\item[\rm (i3)] $\alpha>N-m_1$,  ${\mathcal A}$ is $S$-$m_1$-$C$  and $-\infty<q\leq \frac{\alpha(m_1-1)}{N-m_1}$;

\item[\rm (i4)] $\beta>N-m_2$,  ${\mathcal B}$ is $S$-$m_2$-$C$  and $-\infty<s\leq \frac{\beta(m_2-1)}{N-m_2}$;
\end{enumerate}
then, system \eqref{sys1} has no positive solutions.

\item[\rm (ii)] Assume $\alpha>N-m_1>0$, $\beta>N-m_2>0$, $q>m_1-1$, $s>m_2-1$  and define
$$
\begin{aligned}
\gamma&=\frac{(\alpha+m_1-N)(m_2-1)+(\beta+m_2-N)q}{qs-(m_1-1)(m_2-1)}\,,\\
\xi&=\frac{(\beta+m_2-N)(m_1-1)+(\alpha+m_1-N)s}{qs-(m_1-1)(m_2-1)}.
\end{aligned}
$$
If $$ \max\{ (m_1-1)\gamma-(N-m_1), (m_2-1)\xi -(N-m_2)\}\geq 0, $$ then systems \eqref{sys2} and
\eqref{sys3} have no positive solutions.

\end{enumerate}
\end{theorem}

Our last result concerns the existence of solutions to systems \eqref{sys1}-\eqref{sys3} where
$\mathcal A(x,u, \nabla u) = A(|\nabla u|) \nabla u$ and $\mathcal B(x,u, \nabla u) = B(|\nabla u|)
\nabla u$ for some continuously differentiable functions $A,B$ that satisfy
\begin{equation}\label{ab1}
\frac{tA'(t)}{A(t)} \leq m_1 - 2\,, \quad \frac{tB'(t)}{B(t)} \leq m_2 -2
\qquad \mbox{ for small enough }t>0,
\end{equation}
and
\begin{equation}\label{ab2}
A(t) \geq C_0t^{m_1-2}\,,\quad  B(t) \geq C_0t^{m_2-2} \qquad \mbox{ for small enough }t>0,
\end{equation}
where $C_0>0$ and $m_1,m_2>1$.

\begin{theorem}\label{thm11}
Assume that $\mathcal{A}$, $\mathcal{B}$ satisfy \eqref{ab1} and \eqref{ab2},
 $N>m_1,m_2>1$ and
\begin{equation*}
p+q > m_1 - 1, \qquad r+s > m_2 -1.
\end{equation*}

\begin{itemize}
\item[\rm (i)] If
\begin{align*}
p > \frac{\alpha(m_2-1)}{N-m_2}, \qquad q > \frac{\alpha(m_1-1)}{N-m_1}, \qquad
r > \frac{\beta(m_1-1)}{N-m_1}, \qquad s> \frac{\beta(m_2-1)}{N-m_2}, \\
\frac{q(N-m_1)}{m_1-1} + \frac{p(N-m_2)}{m_2-1} > \alpha + N, \qquad
\frac{r(N-m_1)}{m_1-1} + \frac{s(N-m_2)}{m_2-1} > \beta + N,
\end{align*}
then, \eqref{sys1} has a radial solution in $\Omega= \mathbb R^N$.
\item[\rm (ii)] If
\begin{align*}
p > \frac{\alpha(m_2-1)}{N-m_2}, \qquad q > \frac{\alpha(m_2-1)}{N-m_2}, \qquad
p+q > \frac{(\alpha+N)(m_2-1)}{N-m_2}, \\
r > \frac{\beta(m_1-1)}{N-m_1}, \qquad s> \frac{\beta(m_1-1)}{N-m_1}, \qquad
r+s > \frac{(\alpha+N)(m_1-1)}{N-m_1},
\end{align*}
then, \eqref{sys2} has a radial solution in $\Omega= \mathbb R^N$.
\item[\rm (iii)]
If
\begin{align*}
p > \frac{\alpha(m_1-1)}{N-m_1}, \qquad q > \frac{\alpha(m_2-1)}{N-m_2}, \qquad
r > \frac{\beta(m_1-1)}{N-m_1}, \qquad s> \frac{\beta(m_2-1)}{N-m_2}, \\
\frac{p(N-m_1)}{m_1-1} + \frac{q(N-m_2)}{m_2-1} > \alpha + N, \qquad
\frac{s(N-m_1)}{m_1-1} + \frac{r(N-m_2)}{m_2-1} > \beta + N,
\end{align*}
then, \eqref{sys3} has a radial solution in $\Omega= \mathbb R^N$.
\end{itemize}
\end{theorem}

\medskip

\section{Preliminary Results}

A key tool in our approach is the use of  a priori estimates for solutions $u\in C(\Omega)\cap
W^{1,1}_{loc}(\Omega)$ of the general inequality
\begin{equation}\label{ff}
L_\mathcal{A}  u\geq f(x)\quad\mbox{ in }\Omega,
\end{equation}
where $f\in L^1_{loc}(\Omega)$, $f\geq 0$. Solutions $u$ of \eqref{ff} are understood in the weak
sense, that is, $$ \mathcal{A}(x, u, \nabla u)\in L^{1}_{loc}(\Omega)^{N}, L_{\mathcal{A}}u\in
L^{1}_{loc}(\Omega) $$ and
\begin{equation}\label{var}
\int_{\Omega}\mathcal{A}(x, u, \nabla u)\cdot \nabla \varphi \geq \int_{\Omega}f(x) \varphi
\quad\mbox{ for any }\varphi \in C_{c}^{\infty}(\Omega),\varphi\geq 0.
\end{equation}
The result below is a reformulation of \cite[Proposition 2.1]{BP2001} (see also \cite[Theorem
2.1]{DM2010} for a more general setting).
\begin{proposition}\label{p1}
Let $\Omega\subset \R^N$ be an open set such that for some $R>0$ we have $$ B_{4R}\setminus
B_{R/2}\subset \Omega\qquad  (\mbox{ resp. }B_{2R}\subset \Omega \;). $$ Assume $\mathcal{A}$ is
$W$-$m$-$C$ and let $u\in C(\Omega)\cap  W^{1,1}_{loc}(\Omega)$ be a positive solution of \eqref{ff}.

Let $\phi\in C^\infty_c(\Omega)$ be a standard cut-off function such that:
\begin{itemize}
\item $0\leq \phi\leq 1$ and supp$\,\phi\subset B_{4R}\setminus B_{R/2}$  (resp. supp$\,\phi\subset B_{2R}$);
\item  $\phi=1$ in $B_{2R}\setminus B_{R}$ (resp. $\phi=1$ in $B_{R}$);
\item  $|\nabla \phi|\leq \frac{C}{R}$ in $\Omega$.
\end{itemize}
Then, for any $\lambda>m$,  $0\leq \theta \leq m-1$ and $\ell> m-1-\theta$, there exists $C>0$
independent of $R$ such that
\begin{equation}\label{mainest1}
\int_{\Omega} f(x) u^{-\theta}\phi^\lambda \leq CR^{N-m-\frac{m-1-\theta}{\ell}N}
\Big(\int_\Omega u^\ell\phi^\lambda \Big)^{\frac{m-1-\theta}{\ell}}.
\end{equation}
In particular,
\begin{itemize}
\item[(i)] If $B_{4R}\setminus B_{R/2}\subset \Omega$ then
\begin{equation}\label{mainest3}
\int_{B_{2R}\setminus B_R} f(x) u^{-\theta} \leq CR^{N-m-\frac{m-1-\theta}{\ell}N}
\Big(\int_{B_{4R}\setminus B_{R/2}}u^\ell \Big)^{\frac{m-1-\theta}{\ell}}.
\end{equation}
\item[(ii)] If $B_{2R}\subset \Omega$ then
\begin{equation}\label{mainest4}
\int_{B_{2R}}f(x) u^{-\theta} \leq CR^{N-m-\frac{m-1-\theta}{\ell}N}
\Big(\int_{B_{2R}}u^\ell \Big)^{\frac{m-1-\theta}{\ell}}.
\end{equation}
\end{itemize}
\end{proposition}

\begin{proof}
Assume first $\theta>0$ and let $\varphi=u^{-\theta}\phi^\lambda$ in \eqref{var}. We find $$
\begin{aligned}
\int_{\Omega} f(x) u^{-\theta} \phi^\lambda & \leq
\int_\Omega {\mathcal A}(x, u, \nabla u)\cdot \nabla \varphi \\
&=
-\theta\int_\Omega u^{-\theta-1}\phi^\lambda {\mathcal A}(x, u, \nabla u)\cdot
\nabla u+\lambda\int_\Omega \phi^{\lambda-1}u^{-\theta} {\mathcal A}(x, u, \nabla u)\cdot \nabla \phi.
\end{aligned}
$$
Using Definition \ref{defW} of a $W$-$m$-$C$ operator it follows that
\begin{equation}\label{proof1}
\int_{\Omega}f(x) u^{-\theta} \phi^\lambda +
C_1\theta\int_\Omega u^{-\theta-1}\phi^\lambda |{\mathcal A}(x, u, \nabla u)|^{m'}
\leq \lambda\int_\Omega \phi^{\lambda-1}u^{-\theta} {\mathcal A}(x, u, \nabla u)\cdot \nabla \phi.
\end{equation}
By Young's inequality it follows that
\begin{equation}\label{proof2}
\lambda \phi^{\lambda-1}u^{-\theta} {\mathcal A}(x, u, \nabla u)\cdot \nabla \phi
\leq \frac{C_1\theta}{2} u^{-\theta-1}\phi^\lambda |{\mathcal A}(x, u, \nabla u)|^{m'} +
C(\theta, \lambda)u^{m-1-\theta} \phi^{\lambda-m}|\nabla \phi|^m.
\end{equation}
Using \eqref{proof2} in \eqref{proof1} we find
\begin{equation}\label{proof3}
\int_{\Omega}f(x) u^{-\theta} \phi^\lambda +\frac{C_1\theta}{2}
\int_\Omega u^{-\theta-1}\phi^\lambda |{\mathcal A}(x, u, \nabla u)|^{m'}
\leq C\int_\Omega u^{m-1-\theta} \phi^{\lambda-m}|\nabla \phi|^m.
\end{equation}
If $\theta=m-1$, it follows from \eqref{proof3} and the properties of $\phi$ that $$
\int_{\Omega}f(x) u^{-\theta} \phi^\lambda \leq C\int_\Omega \phi^{\lambda-m}|\nabla \phi|^m\leq
CR^{N-m}, $$ which proves \eqref{mainest1}. Assume next that $0<\theta<m-1$ and let
\begin{equation}\label{g}
\gamma=\frac{\ell}{m-1-\theta}>1.
\end{equation}
From \eqref{proof3} and H\"older's inequality (in the following $\gamma'$ stands for the H\"older's
conjugate of $\gamma$) we find $$
\begin{aligned}
\int_{\Omega}f(x) u^{-\theta} \phi^\lambda&\leq   C\int_\Omega u^{m-1-\theta} \phi^{\lambda-m}|\nabla \phi|^m\\
&\leq C\Big(\int_{{\rm supp}\nabla \phi} u^\ell \phi^\lambda\Big)^{1/\gamma}
\Big(\int_{{\rm supp}\nabla \phi}  \phi^{\lambda-m\gamma'}|\nabla \phi|^{m\gamma'}\Big)^{1/\gamma'}\\
&\leq C R^{\frac{N}{\gamma'}-m} \Big(\int_{{\rm supp}\nabla \phi} u^\ell \phi^\lambda\Big)^{1/\gamma}\\
&=CR^{N-m-\frac{m-1-\theta}{\ell}N} \Big(\int_\Omega u^\ell\phi^\lambda \Big)^{\frac{m-1-\theta}{\ell}}.
\end{aligned}
$$
It remains to discuss the case $\theta=0$. We fix $\beta\in (0, m-1)$ such that
\begin{equation}\label{t}
\tau=\frac{\ell}{(1+\beta)(m-1)}>1
\end{equation}
and let  $\varphi=\phi^\lambda$ in \eqref{var}. By H\"older's inequality we have
\begin{equation}\label{proof4}
\begin{aligned}
\int_{\Omega}f(x) \phi^\lambda & \leq
\lambda\int_\Omega \phi^{\lambda-1} {\mathcal A}(x, u, \nabla u)\cdot \nabla \phi\\
&\leq \lambda \Big(\int_\Omega u^{-\beta-1}\phi^\lambda |{\mathcal A}(x, u, \nabla u)|^{m'} \Big)^{1/m'}
\Big( \int_\Omega u^{(1+\beta)(m-1)} \phi^{\lambda-m}|\nabla \phi|^m\Big)^{1/m}.
\end{aligned}
\end{equation}
Let now $\varphi=u^{-\beta} \phi^\lambda$ in \eqref{var}. Using the same argument as above we
arrive at \eqref{proof3} in which $\theta$ is replaced now with $\beta$. In particular, we find
\begin{equation}\label{proof5}
\int_\Omega u^{-\beta-1}\phi^\lambda |{\mathcal A}(x, u, \nabla u)|^{m'}
\leq C\int_\Omega u^{m-1-\beta} \phi^{\lambda-m}|\nabla \phi|^m.
\end{equation}
Combining \eqref{proof4} and \eqref{proof5} we deduce
\begin{equation}\label{proof6}
\int_{\Omega}f(x)\phi^\lambda \leq \lambda
\Big( \int_\Omega u^{m-1-\beta} \phi^{\lambda-m}|\nabla \phi|^m\Big)^{1/m'}
\Big( \int_\Omega u^{(1+\beta)(m-1)} \phi^{\lambda-m}|\nabla \phi|^m\Big)^{1/m}.
\end{equation}
Using H\"older's inequality with exponents $\gamma$ and $\tau$ defined in \eqref{g} and \eqref{t}
we obtain
\begin{equation}\label{proof7}
\begin{aligned}
\int_\Omega u^{m-1-\beta} \phi^{\lambda-m}|\nabla \phi|^m \!\leq \!
\Big(\int_\Omega u^\ell\phi^\lambda\Big)^{1/\gamma}
\Big(\int_\Omega \phi^{\lambda-m\gamma'}|\nabla \phi|^{m\gamma'}\Big)^{1/\gamma'}\!\!\!\leq
C R^{\frac{N}{\gamma'}-m} \Big(\int_\Omega u^\ell \phi^\lambda\Big)^{1/\gamma},
\end{aligned}
\end{equation}
and
\begin{equation}\label{proof8}
\begin{aligned}
\int_\Omega u^{(1+\beta)(m-1)} \phi^{\lambda-m}|\nabla \phi|^m\!\leq\!
\Big(\int_\Omega u^\ell\phi^\lambda\Big)^{1/\tau}
\Big(\int_\Omega \phi^{\lambda-m\tau'}|\nabla \phi|^{m\tau'}\Big)^{1/\tau'}
\!\!\!\leq C R^{\frac{N}{\tau'}-m} \Big(\int_\Omega  u^\ell \phi^\lambda\Big)^{1/\tau}.
\end{aligned}
\end{equation}
Now, \eqref{mainest1} follows by combining \eqref{proof6}-\eqref{proof8}.
\end{proof}

Letting $\theta= m-1$ in Proposition \ref{p1} above we derive the following a priori estimates
which is a counterpart of \cite[Lemma 4.5]{MV2013}.

\begin{lemma}\label{l1}
Assume $\mathcal{A}$ is $W$-$m$-$C$ for some $m> 1$ and let $u$ be a positive solution of
\eqref{go} in $\R^N\setminus \overline{B}_1$. Then
\begin{equation*}
\Big(\int_{B_{2R}\setminus B_1}u^p\Big)\Big(\int_{B_{2R}\setminus B_R}u^{q-m+1}\Big)
\leq CR^{2N-m-\alpha} \quad\mbox{ for all }R> 2.
\end{equation*}
\end{lemma}
\begin{proof}
Taking $\theta= m-1$ in Proposition \ref{p1} we find
\begin{equation}\label{l1a}
\int_{B_{2R}\setminus B_R}(I_\alpha*u^p)u^{q-m+1} \leq CR^{N-m}  \quad\mbox{ for all  }R>2.
\end{equation}
If $x\in B_{2R}\setminus B_R$ and $y\in B_{2R}\setminus B_{1}$, then $|x-y|\leq |x|+|y|\leq 4R$ so
\begin{equation}\label{l1c}
\begin{aligned}
(I_\alpha*u^p)(x)&\geq C\int_{B_{2R}\setminus B_{1}}\frac{u^{p}(y)}{|x-y|^{N-\alpha}}dy\\
&\geq C\int_{B_{2R}\setminus B_{1}}\frac{u^{p}(y)}{(4R)^{N-\alpha}}dy\\
&\geq CR^{-N+\alpha}\int_{B_{2R}\setminus B_1}u^{p}(y)dy.
\end{aligned}
\end{equation}
The proof concludes by combining \eqref{l1c} and \eqref{l1a}.
\end{proof}

\begin{proposition}\label{p2}{\rm (See \cite[Proposition 2.6]{BP2001})}
Suppose $\mathcal{A}$ is $(H_m)$ and let $u$ be a nonnegative solution of $$ L_\mathcal{A} u=f\geq
0 \quad\mbox{ in }\R^N\setminus \overline{B}_1, $$ for some $f\in L^1_{loc}(\R^N\setminus
\overline{B}_1)$. Then, there exists $c> 0$ such that
\begin{equation*}
\left\{
\begin{aligned}
u(x)&\geq c|x|^{-\frac{N-m}{m-1}} &&\quad\mbox{ if }\,N> m\\
u(x)&\geq c &&\quad\mbox{ if }\,N\leq m
\end{aligned}
\right.
\qquad\mbox{ in } \R^N\setminus B_2.
\end{equation*}
\end{proposition}

\begin{proposition}\label{p2.1}{\rm (See \cite[Proposition 2.7(ii)]{BP2001})}
Suppose $N>m>1$, $\mathcal{A}$ is $(H_m)$ and let $u>0$  satisfy $$ L_\mathcal{A} u\geq C|x|^{-N}
\quad\mbox{ in }\R^N\setminus\overline{B}_1, $$ for some $C>0$. Then, there exists $c> 0$ such that
$$ u(x) \geq c|x|^{-\frac{N-m}{m-1}} \big(\ln |x|\big)^{\frac{1}{m-1}} \quad\mbox{ in }
\R^N\setminus B_2. $$
\end{proposition}

The next result concerns the inequality
\begin{equation}\label{p3a}
L_\mathcal{A} u= -{\rm div}[\mathcal{A}(x,u, \nabla u)]\geq |x|^{\sigma}u^q
\quad\mbox{ in }\R^N\setminus \overline{B}_2.
\end{equation}

\begin{proposition}\label{p3}{\rm (See \cite[Theorems 3.3 and 3.4]{BP2001})}
Suppose $N> m> 1$, $\mathcal {A}$ is $(H_m)$ and one of the following conditions hold.
\begin{enumerate}
\item[\rm (i)] $m-1<q\leq \frac{(N+\sigma)(m-1)}{N-m}$.

\item[\rm (ii)] ${\mathcal A}$ is $S$-$m$-$C$ and $-\infty<q\leq m-1$.
\end{enumerate}
Then \eqref{p3a} has no positive solutions.
\end{proposition}
Recall that if ${\mathcal A}$ is $(H_m)$ then ${\mathcal A}$ is $S$-$m$-$C$ if and only if
\eqref{smc} holds.

\medskip

Our next result concerns systems of type
\begin{equation}\label{sysh}
\left\{
\begin{aligned}
L_\mathcal{A} u & = -{\rm div}[\mathcal{A}(x, u, \nabla u)]\geq |x|^av^q\\
L_\mathcal{B} v& = -{\rm div}[\mathcal{B}(x, v, \nabla v)]\geq |x|^bu^s
\end{aligned}
\right.
\qquad\mbox{ in } \R^N\setminus B_{1}
\end{equation}
where ${\mathcal A}$ and ${\mathcal B}$ are $(H_{m_1})$ and $(H_{m_2})$ respectively with
$m_1,m_2>1$. Also, $$ a,b\in \R\,, \;q>m_1-1\,, \; s>m_2-1. $$ Denote $$
\gamma=\frac{(a+m_1)(m_2-1)+(b+m_2)q}{qs-(m_1-1)(m_2-1)}\,,\quad
\xi=\frac{(b+m_2)(m_1-1)+(a+m_1)s}{qs-(m_1-1)(m_2-1)}. $$

\begin{proposition}\label{psys}{\rm (See \cite[Theorem 5.3]{BP2001})}
Suppose $N>\max\{ m_1, m_2\}$  and $$ \max\{ (m_1-1)\gamma-(N-m_1), (m_2-1)\xi -(N-m_2)\}\geq 0. $$
Then, \eqref{sysh} has no positive solutions.
\end{proposition}

Several times in this paper we shall make use of the following lemma which provides basic estimates
for Riesz potentials. We state it here to avoid repetitive arguments in our proofs.
\begin{lemma}\label{lbas}
Let $\alpha\in (0,N)$.

\begin{enumerate}
\item[\rm (i)]
If $f\in L^1_{loc}(\R^N\setminus \overline{B}_1)$, $f\geq 0$, then there exists $C>0$ such that
$$ (I_\alpha*f)(x)\geq C|x|^{\alpha-N}\quad\mbox{ for any }x\in \R^N\setminus B_2. $$

\item[\rm (ii)]
If $f(x)\geq c|x|^{-\gamma}$ in $\R^N\setminus B_2$ for some $c,\gamma>0$, then, given $p>0$ we have
$$ \left\{
\begin{aligned}
&  (I_\alpha*f^p)(x)=\infty && \quad\mbox{ if } \; p\gamma \leq \alpha \\
&  (I_\alpha*f^p)(x)\geq C|x|^{\alpha-p\gamma} && \quad\mbox{ if } \; p\gamma> \alpha
\end{aligned}
\right. \quad\mbox{ in }\; \R^N\setminus B_2.
$$

\item[\rm (iii)]
If $f\in L^1(\R^N)\cap C(\R^N)$ and $\displaystyle \limsup_{|x|\to \infty}f(x)|x|^\beta<\infty$
for some $\beta>\alpha$, then $$ \left\{
\begin{aligned}
&\limsup_{|x|\to \infty} |x|^{\beta-\alpha} (I_\alpha\ast f)(x)<\infty &&\quad\mbox{ if }\; \alpha<\beta<N,\\
&\limsup_{|x|\to \infty} \frac{|x|^{N-\alpha}}{\log|x|} (I_\alpha\ast f)(x)<\infty  &&\quad\mbox{ if }\; \beta=N,\\
&\limsup_{|x|\to \infty} |x|^{N-\alpha} (I_\alpha\ast f)(x)<\infty&&\quad\mbox{ if }\; \beta>N.
\end{aligned}
\right.
$$
\end{enumerate}
\end{lemma}
\begin{proof}
(i) For any $x\in \R^N\setminus B_2$ we have $$ (I_\alpha*f)(x)\geq C\int_{3/2< |y|<
2}\frac{f(y)}{|x-y|^{N-\alpha}}dy \geq C\int_{3/2< |y|< 2}\frac{f(y)}{|2x|^{N-\alpha}}dy =
\frac{C}{|x|^{N-\alpha}}. $$ (ii) We have $$ (I_\alpha*f^p)(x)\geq C\int_{|y|\geq
2|x|}\frac{|y|^{-p\gamma}}{|x-y|^{N-\alpha}}dy \geq C\int_{|y|\geq 2|x|} |y|^{\alpha-N-p\gamma} dy=
C\int_{2|x|}^{\infty}t^{\alpha-p \gamma}\frac{dt}{t}, $$ and the conclusion follows.

(iii) This was proved in \cite[Lemma A.1]{MV2013}.
\end{proof}

\section{Proof of Theorem \ref{thm2}}

(i) Assume first that $q= m-1$ which also implies $\alpha> N-m$. Then by Lemma \ref{l1} we deduce
$$ \int_{B_{2R}\setminus B_1}u^p dx\leq CR^{N-m-\alpha} \quad\mbox{ for all }R> 2. $$ Letting
$R\to \infty$ in the above estimate and using the fact that $\alpha> N-m$ it follows that
\eqref{go} has no positive solutions.

Assume next that $q< m-1$. By H\"older's inequality we estimate
\begin{equation}\label{pr1}
\int_{B_{2R}\setminus B_R} 1 \leq \Big(\int_{B_{2R}\setminus B_R}u^p\Big)^{\frac{m-1-q}{p+m-1-q}}
\Big(\int_{B_{2R}\setminus B_R}u^{q-m+1}\Big)^{\frac{p}{p+m-1-q}},
\end{equation}
which we rewrite as
\begin{equation*}
CR^N\leq \Big[\Big(\int_{B_{2R}\setminus B_R}u^p\Big)\Big(\int_{B_{2R}\setminus B_R}u^{q-m+1}\Big)
\Big]^{\frac{m-1-q}{p+m-1-q}}\Big(\int_{B_{2R}\setminus B_R}u^{q-m+1}\Big)^{\frac{p-m+1+q}{p+m-1-q}}.
\end{equation*}
Now, by Lemma \ref{l1} we deduce
\begin{equation*}
CR^N\leq C\big(R^{2N-m-\alpha}\big)^{\frac{m-1-q}{p+m-1-q}}
\Big(\int_{B_{2R}\setminus B_R}u^{q-m+1}\Big)^{\frac{p-m+1+q}{p+m-1-q}},
\end{equation*}
which yields
\begin{equation}\label{pr2}
\int_{B_{2R}\setminus B_R}u^{q-m+1} \geq CR^{N+\frac{(m+\alpha)(m-1-q)}{p-m+1+q}} \quad\mbox{ for }R> 2.
\end{equation}
Again by Lemma \ref{l1} we have
\begin{equation}\label{pr3}
\Big(\int_{B_{2R}\setminus B_1}u^{p}\Big)\Big(\int_{B_{2R}\setminus B_R}u^{q-m+1}\Big)
\leq CR^{2N-m-\alpha} \quad\mbox{ for all }R> 2.
\end{equation}
Therefore,
\begin{equation}\label{pr4}
\int_{B_{2R}\setminus B_R}u^{q-m+1}dx\leq \frac{CR^{2N-m-\alpha}}{\int_{B_{2R}\setminus B_1}u^{p}}
\leq \frac{CR^{2N-m-\alpha}}{\int_{B_{4}\setminus B_1}u^{p} }\leq CR^{2N-m-\alpha}.
\end{equation}
From \eqref{pr2} and \eqref{pr4} we deduce
\begin{equation*}
C_1 R^{N+\frac{(m+\alpha)(m-1-q)}{p-m+1+q}}\leq \int_{B_{2R}\setminus B_R}u^{q-m+1}
\leq C_2 R^{2N-m-\alpha} \quad\mbox{ for all }R> 2.
\end{equation*}
Since  $q< m-1-\frac{N-\alpha-m}{N}p$, the above inequality cannot hold for large $R> 2$. Hence,
\eqref{go} cannot have positive solutions.

(ii)  Assume $0< q= m-1 - \frac{N-m-\alpha}{N}p$. With a  similar approach as above we have
\begin{equation}\label{ee0}
C_1 R^{2N-m-\alpha} \leq \int_{B_{2R} \backslash B_R} u^{q-m+1}
\leq C_2 R^{2N-m-\alpha} \quad\mbox{ for all }R> 2.
\end{equation}
Since $$ \frac{p}{p+m-1-q} = \frac{N}{2N-m-\alpha}, $$ estimate \eqref{pr1} yields
\begin{equation*}
CR^N \leq \left( \int_{B_{2R} \backslash B_R} u^p \right)^{\frac{m-1-q}{p+m-1-q}}
\cdot CR^N \quad\mbox{ for all }R> 2.
\end{equation*}
Hence
\begin{equation*}
\int_{B_{2R} \backslash B_R} u^p \geq C \quad\mbox{ for all }R> 2,
\end{equation*}
which shows that $\int_{\R^N\setminus B_1} u^p = \infty$ and then $\int_{B_{2R}\setminus B_1} u^p
\rightarrow \infty$ as $r\rightarrow \infty$. Further, from \eqref{pr4} we have
\begin{equation*}
\int_{B_{2R}\setminus B_R}u^{q-m+1} \leq \frac{CR^{2N-m-\alpha}}{\int_{B_{2R}\setminus B_1}u^{p}}=
o(R^{2N-\alpha-m}) \quad\mbox{ as } R\rightarrow \infty,
\end{equation*}
which contradicts the first estimate in \eqref{ee0}.

(iii)  Assume first that $p+q= m-1$. By Lemma \ref{l1} we have
\begin{equation}\label{pq1}
\Big(\int_{B_{2R}\setminus B_1}u^p\Big)\Big(\int_{B_{2R}\setminus B_R}u^{-p}\Big)
\leq CR^{2N-m-\alpha} \quad\mbox{ for all }R> 2.
\end{equation}
On the other hand, by H\"older's inequality we deduce
\begin{equation}\label{pq1.2}
CR^{2N}= \Big(\int_{B_{2R}\setminus B_R} 1 \Big)^2\leq \Big(\int_{B_{2R}\setminus B_R}u^p\Big)
\Big(\int_{B_{2R}\setminus B_R}u^{-p}\Big) \quad\mbox{ for all }R> 2.
\end{equation}
Now, \eqref{pq1} and \eqref{pq1.2} cannot hold for $R>2$ sufficiently large. This shows that
\eqref{go} cannot have positive solutions.

\medskip

Assume now that $p+q< m-1$. We apply H\"older inequality to derive \eqref{pr1} which we may rewrite
as
\begin{equation*}
CR^N\leq \Big(\int_{B_{2R}\setminus B_R}u^p\Big)^{\frac{m-1-p-q}{p+m-1-q}}
\Big[\Big(\int_{B_{2R}\setminus B_R}u^{p}\Big)
\Big(\int_{B_{2R}\setminus B_R}u^{q-m+1}\Big)\Big]^{\frac{p}{p+m-1-q}}.
\end{equation*}
Using the estimate in Lemma \ref{l1}  we find
\begin{equation*}
R^N \leq C\Big(R^{2N-m-\alpha}\Big)^{\frac{p}{p+m-1-q}}
\Big(\int_{B_{2R}\setminus B_R}u^p\Big)^{\frac{m-1-p-q}{p+m-1-q}} \quad\mbox{ for all }R> 2,
\end{equation*}
 which implies
\begin{equation}\label{pr5}
\int_{B_{2R}\setminus B_R}u^p \geq CR^{N+\frac{p(m+\alpha)}{m-1-p-q}} \quad\mbox{ for all }R> 2.
\end{equation}
On the other hand, equation \eqref{c1} yields
\begin{equation*}
\int_{|x|> 1}\frac{u^{p}(x)}{|x|^{N-\alpha}}dx< \infty,
\end{equation*}
which further implies
\begin{equation*}
R^{\alpha-N}\int_{B_{2R}\setminus B_R}u^p \leq C \quad\mbox{ for all }R> 2.
\end{equation*}
Combining this with \eqref{pr5} we find
\begin{equation*}
C_1 R^{N-\alpha}\geq \int_{B_{2R}\setminus B_R}u^p
\geq C_2 R^{N+\frac{p(m+\alpha)}{m-1-p-q}} \quad\mbox{ for all }R> 2,
\end{equation*}
which gives a contradiction as $N-\alpha< N+\frac{p(m+\alpha)}{m-1-p-q}$.

\qed

\section{Proof of Theorem \ref{thm1}}

Assume by contradiction that $u$ is a positive solution of \eqref{go}.

(i)  By Proposition \ref{p2} we have $u\geq c$ in $\R^N\setminus B_2$, for some constant $c> 0$.
Therefore,
\begin{equation*}
\int_{|y|> 2}\frac{u^{p}(y)}{1+|y|^{N-\alpha}}dy\geq c\int_{|y|> 2}\frac{dy}{2|y|^{N-\alpha}}= \infty,
\end{equation*}
which contradicts \eqref{c1}.

(ii1) By Proposition \ref{p2} we deduce $u\geq c|x|^{-\frac{N-m}{m-1}}$ in $\R^N\setminus B_2$, for
some constant $c> 0$. Since $p\leq \frac{\alpha(m-1)}{N-m}$, by Lemma \ref{lbas}(ii) we have
$(I_\alpha*u^p)(x)=\infty$ for all $x\in \R^N\setminus B_2$, which contradicts the fact that
$(I_\alpha*u^p)u^q \in L^{1}_{loc}(\Omega)$.

(ii2) Assume $m-1< q\leq \frac{\alpha(m-1)}{N-m}$ which further yields $\alpha> N-m$. By Lemma
\ref{lbas}(i) we have $$ (I_\alpha*u^p)(x)\geq C|x|^{\alpha-N}\quad\mbox{ in }\R^N\setminus B_2. $$
Thus, $u$ satisfies
\begin{equation*}
L_{\mathcal{A}}u\geq C|x|^{\alpha-N}u^q \quad\mbox{ in } \R^N \setminus B_2.
\end{equation*}
We further apply Proposition \ref{p3}(i) with $\sigma= \alpha-N$. Thus, if $$ m-1<q\leq
\frac{(N+\sigma)(m-1)}{N-m}= \frac{\alpha(m-1)}{N-m}, $$ then \eqref{go} has no positive solutions.

(ii3) Assume first that $p+q< \frac{(N+\alpha)(m-1)}{N-m}$. Proceeding as in \cite[Proposition 4.6]{MV2013} we
use H\"older inequality to estimate
\begin{equation*}
\Big(\int_{B_{2R}\setminus B_R}u^p\Big)\Big(\int_{B_{2R}\setminus B_R}u^{q-m+1}\Big)\geq
\Big(\int_{B_{2R}\setminus B_R}u^{\frac{p+q-m+1}{2}}\Big)^2 \quad\mbox{ for all } R> 2.
\end{equation*}
Using Lemma \ref{l1} together with the estimate $u\geq C|x|^{-\frac{N-m}{m-1}}$ in $\R^N\setminus
B_2$ we find
\begin{equation*}
C_1 R^{2N-m-\alpha}\geq \Big(\int_{B_{2R}\setminus B_R}u^{\frac{p+q-m+1}{2}}\Big)^2
\geq C_2 R^{2N-\frac{(p+q-m+1)(N-m)}{m-1}} \quad\mbox{ for all }R> 2.
\end{equation*}
However, this is a contradiction as $2N-m-\alpha< 2N-\frac{(p+q-m+1)(N-m)}{m-1}$.

\medskip

Assume now that $p+q= \frac{(N+\alpha)(m-1)}{N-m}$. If $q<0$, the nonexistence of a positive solution follows from Theorem \ref{thm2}(i). Hence, it remains to discuss the case $q\geq 0$.
By Proposition \ref{p2} we deduce $u\geq
c|x|^{-\frac{N-m}{m-1}}$ in $\R^N\setminus B_2$, for some constant $c>0$. If $p\frac{N-m}{m-1}\leq
\alpha$, by Lemma \ref{lbas}(ii) we find $(I_\alpha*u^p)=\infty$ in $\R^N\setminus B_1$,
contradiction. Assume in the following that $p\frac{N-m}{m-1}> \alpha$. By the last estimate in
Lemma \ref{lbas}(ii) it follows that $$ (I_\alpha*u^p)(x) \geq  C
|x|^{\alpha-p\frac{N-m}{m-1}}\quad\mbox{ in } \R^N\setminus B_2. $$ Thus, $u$ satisfies $$
L_{\mathcal{A}}u\geq (I_\alpha*u^p)u^q\geq C|x|^{\alpha-(p+q)\frac{N-m}{m-1}}=C|x|^{-N} \quad\mbox{
in } \R^N \setminus B_2. $$ We now use Proposition \ref{p2.1} to deduce
\begin{equation}\label{logg}
u(x) \geq c|x|^{-\frac{N-m}{m-1}} \big(\ln |x|\big)^{\frac{1}{m-1}} \quad\mbox{ in } \R^N\setminus B_2,
\end{equation}
for some $c>0$. Finally, by \eqref{logg}, Lemma \ref{l1} and H\"older's inequality, for $R>2$ we
have $$
\begin{aligned}
CR^{2N-m-\alpha}&\geq\Big(\int_{B_{2R}\setminus B_R}u^p\Big)\Big(\int_{B_{2R}\setminus B_R}u^{q-m+1}\Big)\\
&\geq \Big(\int_{B_{2R}\setminus B_R}u^{\frac{p+q-m+1}{2}}\Big)^2 \\
&\geq CR^{2N-m-\alpha}\big(\ln R\big)^{\frac{m+\alpha}{N-m}},
\end{aligned}
$$
which is a contradiction for $R>2$ large.

\section{Proof of Theorem \ref{thm3}}

If $u=u(|x|)=u(r)$ is a radial function, one has
\begin{align} \label{iden}
\text{div} ( A(|\nabla u|) \nabla u)
= u''(r) \cdot \Bigl[ A'(|u'(r)|) \cdot |u'(r)| + A(|u'(r)|) \Bigr]
+ \frac{N-1}{r} \cdot A(|u'(r)|) \cdot u'(r).
\end{align}

We look for radial solutions of the form $u(r)= \e (1+r)^{-\gamma}$ for some $\e,\gamma>0$ to be
chosen later. Note that $u(r)$ is decreasing and let $t=|u'(r)|$ for convenience.  According to
\eqref{iden} we have
\begin{align*}
\text{div} ( A(|\nabla u|) \nabla u)
&= u''(r) \cdot \Bigl[ t A'(t) + A(t) \Bigr] - \frac{N-1}{r} \cdot tA(t) \\
&= \frac{tA(t)}{r} \cdot \left[ -\frac{ru''(r)}{u'(r)} \cdot \frac{t A'(t) + A(t)}{A(t)} + 1-N \right] \\
&= \frac{tA(t)}{r} \cdot \left[ \frac{(\gamma+1)r}{r+1} \cdot \frac{t A'(t) + A(t)}{A(t)} + 1-N \right].
\end{align*}
Since $t = \e\gamma (1+r)^{-\gamma-1} \leq \e\gamma$, it is clear that $t$ converges to zero
uniformly as $\e\to 0$.  If we now assume that $\e$ is sufficiently small, our assumption
\eqref{upper} becomes applicable and
\begin{align*}
\text{div} ( A(|\nabla u|) \nabla u)
&\leq \frac{tA(t)}{r} \cdot \left[ (\gamma+1)(m-1) + 1-N \right] \\
&= \frac{tA(t)}{r} \cdot \left[ \gamma(m-1) + m-N \right].
\end{align*}
As long as $0<\gamma<\frac{N-m}{m-1}$, the constant in the square brackets is negative, so we get
\begin{equation} \label{est1}
L_{\mathcal{A}}u= -\text{div} ( A(|\nabla u|) \nabla u) \geq \frac{C_0 tA(t)}{r} \geq
\frac{C_1 t^{m-1}}{1+r} = C_2 \e^{m-1} (1+r)^{-\gamma(m-1)-m}
\end{equation}
for any $0<\gamma<\frac{N-m}{m-1}$ and all sufficiently small $\e>0$.

When it comes to part (i), one may choose the parameter $\gamma>0$ so that
\begin{equation} \label{gamma1}
\frac Np < \gamma < \frac{N-m}{m-1}.
\end{equation}
Since $u(r)= \e (1+r)^{-\gamma}$ with $\gamma p> N > \alpha$, it follows by Lemma \ref{lbas}(iii)
that
\begin{equation} \label{est2}
(I_\alpha \ast u^p) \cdot u^q \leq C_3 \e^p (1+r)^{\alpha - N} \cdot u^q =
C_4 \e^{p+q} (1+r)^{\alpha - N -\gamma q}.
\end{equation}
Using our estimate \eqref{est1} and the fact that $q=m-1$, we also have
\begin{equation*}
L_{\mathcal{A}}u \geq C_2 \e^{m-1} (1+r)^{-\gamma(m-1)-m} =
C_2 \e^{m-1} (1+r)^{-\gamma q -m}.
\end{equation*}
Since $\alpha = N-m$ and $p+q>m-1$, one may then combine the last two estimates to get
\begin{align*}
L_{\mathcal{A}}u
&\geq C_2 \e^{m-1} (1+r)^{-\gamma q -m} \\
&\geq C_4 \e^{p+q} (1+r)^{-\gamma q + \alpha-N} \geq (I_\alpha \ast u^p) \cdot u^q
\end{align*}
for all sufficiently small $\e>0$.  This completes the proof of part (i).

When it comes to part (ii), one may choose the parameter $\gamma>0$ so that
\begin{equation} \label{gamma2}
\max \left\{ \frac{\alpha + m}{p+q-m+1}, \frac{\alpha}{p} \right\} < \gamma < \frac{N}{p} \leq \frac{N-m}{m-1}.
\end{equation}
Since $\gamma<\frac{N-m}{m-1}$, our estimate \eqref{est1} is still valid.  Since $\alpha < \gamma p
<N$, however, one has
\begin{equation} \label{est3}
(I_\alpha \ast u^p) \cdot u^q \leq C_3 \e^p (1+r)^{\alpha - \gamma p} \cdot u^q
= C_4 \e^{p+q} (1+r)^{\alpha - \gamma (p+q)}
\end{equation}
by Lemma \ref{lbas}(iii).  Using our estimate \eqref{est1} along with \eqref{gamma2}, we conclude
that
\begin{align} \label{est4}
L_{\mathcal{A}}u
&\geq C_2 \e^{m-1} (1+r)^{-\gamma(m-1)-m} \notag \\
&\geq C_4 \e^{p+q} (1+r)^{\alpha - \gamma (p+q)} \geq (I_\alpha \ast  u^p) \cdot u^q
\end{align}
for all small enough $\e>0$, since $p+q>m-1$.  This completes the proof of part (ii).

\medskip

\section{Proof of Theorem \ref{thm4}}
We proceed as in the previous theorem, this time  we look for a solution of the form
\begin{equation} \label{u|def}
u(r) = \e (1+r)^{-\gamma} \cdot \left( 1 - \frac{k}{1+ \log (1+r)} \right),
\end{equation}
where $\gamma= \frac{N-m}{m-1}>0$ and $\e,k>0$ are sufficiently small.  As one can easily check,
\begin{align} \label{d1|eq}
u'(r) &= - \frac{\gamma u(r)}{1+r} + \frac{\e k (1+r)^{-\gamma -1} }{(1 + \log (1+r))^2} \\
&= -\e \cdot \frac{\gamma (\log (1+r))^2 + \gamma (2-k) \log (1+r) +
    [\gamma - k(\gamma +1)]}{(1+r)^{\gamma +1} \cdot (1 + \log (1+r))^2}. \notag
\end{align}
If we take $0<k < \frac{\gamma}{\gamma +1}$, then the numerator is the sum of non-negative terms,
so $u(r)$ is decreasing and we also have the estimate
\begin{equation} \label{d1|est}
\frac{C_0\e}{(1+r)^{\gamma +1} \cdot (1 + \log (1+r))^2} \leq |u'(r)| \leq \frac{C_1 \e}{(1+r)^{\gamma +1}}
\end{equation}
for some fixed constants $C_0,C_1>0$.  Next, we differentiate \eqref{d1|eq} to find that
\begin{align} \label{d2|eq}
u''(r) + \frac{\gamma u'(r)}{1+r} &= \frac{\gamma u(r)}{(1+r)^2}
- \frac{\e k(\gamma +1) (1+r)^{-\gamma-2} }{(1+ \log (1+r))^2}
- \frac{2\e k (1+r)^{-\gamma-2} }{(1+ \log (1+r))^3}.
\end{align}
Letting $z= \log (1+r)$ for convenience, one may express the last equation in the form
\begin{align*}
u''(r) + \frac{\gamma u'(r)}{1+r}
&= \frac{\gamma \e (1-k+z)}{(1+r)^{\gamma+2}(1+z)} - \frac{\e k(\gamma +1)}{(1+r)^{\gamma+2} (1+z)^2} -
\frac{2\e k}{(1+r)^{\gamma+2} (1+z)^3} \\
&= \e\cdot \frac{\gamma z^3 + \gamma (3-k) z^2 + (3\gamma - 3\gamma k - k) z +
    (\gamma - 2\gamma k - 3k)}{(1+r)^{\gamma+2} (1+z)^3}.
\end{align*}
If we take $0< k < \frac{\gamma}{2\gamma +3} < \frac{3\gamma}{3\gamma+1}$, then the numerator is
the sum of non-negative terms.  Assuming that $k>0$ is sufficiently small, we thus have
\begin{equation}
u''(r) \geq -\frac{\gamma u'(r)}{1+r} \geq 0.
\end{equation}
On the other hand, one may combine equations \eqref{u|def}, \eqref{d1|eq} and \eqref{d2|eq} to find
that
\begin{equation} \label{u|ode}
u''(r) + (\gamma +1) \cdot \frac{u'(r)}{1+r} =
-\frac{\e k(\gamma\log (1+r) + \gamma + 2)}{(1+r)^{\gamma +2} \cdot (1 + \log (1+r))^3} \leq 0.
\end{equation}

Next, we employ our identity \eqref{iden} which gives
\begin{align*}
\text{div} ( A(|\nabla u|) \nabla u)
&= u''(r) \cdot \Bigl[ t A'(t) + A(t) \Bigr] + \frac{N-1}{r} \cdot A(t) \cdot u'(r) \\
&= A(t) \cdot \left[ u''(r) \cdot \frac{t A'(t) + A(t)}{A(t)} + \frac{N-1}{r} \cdot u'(r) \right]
\end{align*}
with $t= |u'(r)|$.  If we let $\e$ approach zero, then $t$ converges to zero uniformly by
\eqref{d1|est} and our assumption \eqref{upper} becomes applicable.  Since $u''(r)\geq 0$ and
$u'(r)\leq 0$, we get
\begin{equation}\label{t1}
\text{div} ( A(|\nabla u|) \nabla u)
\leq (m-1) A(t) \cdot \left[ u''(r) + \frac{N-1}{m-1} \cdot \frac{u'(r)}{1+r} \right].
\end{equation}
Since $\gamma = \frac{N-m}{m-1}$, we have $\frac{N-1}{m-1} = \gamma + 1$ and  \eqref{u|ode}
ensures that
\begin{equation}\label{t2}
\begin{aligned}
u''(r) + \frac{N-1}{m-1} \cdot \frac{u'(r)}{1+r}
&= -\frac{\e k(\gamma\log (1+r) + \gamma + 2)}{(1+r)^{\gamma +2} \cdot (1 + \log (1+r))^3} \\
&\leq -\frac{C_2 \e}{(1+r)^{\gamma +2} \cdot (1 + \log (1+r))^2}
\end{aligned}
\end{equation}
for some fixed constant $C_2>0$.  In view of \eqref{t1} and \eqref{t2}, we  have
\begin{align*}
L_{\mathcal{A}}u&=-\text{div} ( A(|\nabla u|) \nabla u)\\
&\geq \frac{C_2 \e (m-1) A(t)}{(1+r)^{\gamma +2} \cdot (1 + \log (1+r))^2} \\
&\geq \frac{C_3 \e \cdot |u'(r)|^{m-1}}{(1+r)^{\gamma +2} \cdot (1 + \log (1+r))^2 \cdot |u'(r)|}
\end{align*}
for some fixed constant $C_3>0$.  Using the two estimates in \eqref{d1|est}, we conclude that
\begin{align} \label{est5}
L_{\mathcal{A}}u
\geq \frac{C_4 \e^{m-1} (1+r)^{-\gamma(m-1) -m} }{(1 + \log (1+r))^{2m}} =
\frac{C_4 \e^{m-1} (1+r)^{-N}}{(1 + \log (1+r))^{2m}}
\end{align}
for all sufficiently small $\e>0$ because $\gamma= \frac{N-m}{m-1}$ by above.

In order to estimate the convolution term, we note that $\gamma p = \frac{p(N-m)}{m-1} > \alpha$ by
assumption and that $u(r) \leq \e (1+r)^{-\gamma}$ by \eqref{u|def}.  According to Lemma
\ref{lbas}(iii), this implies
\begin{equation} \label{est6}
(I_\alpha \ast u^p) u^q \leq
\left\{ \begin{array}{lc}
C_5 \e^{p+q} (1+r)^{\alpha - \gamma(p+q)} &\text{if \;}\,\gamma p<N, \\
C_5 \e^{p+q} (1+r)^{\alpha -N -\gamma q} (1+\log (1+r)) &\text{if \;}\,\gamma p=N, \\
C_5 \e^{p+q} (1+r)^{\alpha -N -\gamma q} &\text{if \;}\, \gamma p>N.
\end{array} \right.
\end{equation}
Moreover, the three conditions in \eqref{three} can also be expressed in the form
\begin{equation}
\gamma p > \alpha, \qquad \gamma q > \alpha \quad\mbox{ and }\quad  \gamma(p+q) > \alpha + N.
\end{equation}

\noindent \textbf{Case 1:}  $\gamma p < N$. Then, the estimates \eqref{est5} and \eqref{est6} yield
\begin{align*}
L_{\mathcal{A}}u
&\geq \frac{C_4 \e^{m-1} (1+r)^{-N}}{(1 + \log (1+r))^{2m}} \\
&\geq C_5 \e^{p+q} (1+r)^{\alpha - \gamma(p+q)} \geq (I_\alpha \ast u^p) u^q
\end{align*}
for all sufficiently small $\e>0$ because $\gamma(p+q) > \alpha +N$ and $p+q>m-1$.

\noindent \textbf{Case 2:} $\gamma p \geq N$. In this situation, estimates \eqref{est5} and \eqref{est6} imply
\begin{align*}
L_{\mathcal{A}}u
&\geq \frac{C_4 \e^{m-1} (1+r)^{-N}}{(1 + \log (1+r))^{2m}} \\
&\geq C_5 \e^{p+q} (1+r)^{\alpha - N - \gamma q} \cdot (1+ \log(1+r))
\geq (I_\alpha \ast u^p) u^q
\end{align*}
for all sufficiently small $\e>0$ because $\gamma q > \alpha$ and $p+q>m-1$.

In either case then, the function $u$ given in \eqref{u|def} is a bounded radial solution of
\eqref{go}.

\medskip

\section{Proofs of Theorems \ref{thm5} - \ref{thm8}}

\noindent{\bf Proof of Theorem \ref{thm5}.}
We assume that $\Omega$ is contained in the ball $B_R$ for some $R>0$ and we look for solutions
of the form $u(r) = \delta(1 + R-r)$ for some small enough $\delta>0$.  According to
\eqref{iden},
\begin{align*}
\text{div} ( A(|\nabla u|) \nabla u) = \frac{N-1}{r} \cdot A(|u'(r)|) \cdot u'(r)
= -\frac{N-1}{r} \cdot \delta A(\delta).
\end{align*}
As long as $\delta>0$ is sufficiently small, one may employ \eqref{con1} to find
\begin{equation} \label{est7}
L_{\mathcal{A}}u \geq \frac{N-1}{r} \cdot C_0 \delta^{m-1} \geq
C_1 \delta^{m-1}
\end{equation}
for all $r<R$.  On the other hand, $u(r)\leq \delta(1+R)$ by definition, so we also have
\begin{equation*}
I_\alpha \ast u^p \leq C_\alpha \int_{|y|<R} \frac{u(y)^p \,dy}{|x-y|^{N-\alpha}}
\leq C_2 \delta^p \int_{|y|<R} \frac{dy}{|x-y|^{N-\alpha}}.
\end{equation*}
Letting $I_1$ be the part of the integral with $|x-y|\leq |y|$ and $|y|<R$, we have
\begin{equation*}
I_1 \leq C_2 \delta^p \int_{|x-y|<R} \frac{dy}{|x-y|^{N-\alpha}} = C_3 \delta^p.
\end{equation*}
Letting $I_2$ be the remaining part with $|x-y|\geq |y|$ and $|y|<R$, we similarly have
\begin{equation*}
I_2 \leq C_2 \delta^p \int_{|y|<R} \frac{dy}{|y|^{N-\alpha}} = C_3 \delta^p.
\end{equation*}
We now combine the last two estimates to derive
\begin{equation*}
(I_\alpha \ast u^p) \cdot u^q \leq 2C_3 \delta^p \cdot u^q =
2C_3 \delta^{p+q} (1+R-r)^q \leq C_4 \delta^{p+q}
\end{equation*}
for all $r<R$.  Since $p+q>m-1$ by assumption, it now follows by \eqref{est7} that
\begin{equation*}
L_{\mathcal{A}}u \geq C_1 \delta^{m-1} \geq C_4 \delta^{p+q} \geq
(I_\alpha \ast u^p) \cdot u^q
\end{equation*}
for all $r<R$ and all sufficiently small $\delta>0$.  This also completes the proof.

\medskip

\noindent{\bf Proof of Theorem \ref{thm6}.}
We assume that $\Omega$ is contained in the interval $(0,R)$ and we look for solutions of the form
$u(r) = \delta \log(1+R+r)$ for some small enough $\delta>0$.  Since $N=1$, we have
\begin{align*}
\text{div} ( A(|\nabla u|) \nabla u)
&= u''(r) \cdot \bigl[ t A'(t) + A(t) \bigr]
= -\frac{\delta}{(1+R+r)^2} \cdot \bigl[ t A'(t) + A(t) \bigr],
\end{align*}
where $t = |u'(r)| = \delta/(1+R+r)$.  Note that $t$ converges to zero uniformly in $\Omega$ as
$\delta \to 0$.  Assuming that $\delta$ is sufficiently small, we may thus conclude that
\begin{equation*}
L_{\mathcal{A}}u=-\text{div} ( A(|\nabla u|) \nabla u) \geq \frac{\delta \cdot C_0 \delta^{m-2}}{(1+R+r)^2}
\geq C_1 \delta^{m-1}
\end{equation*}
for all $r<R$.  On the other hand, $u(r)\leq \delta \log(1+2R)$ by definition, so we also have
\begin{equation*}
I_\alpha \ast u^p \leq C_2 \int_{|y|<R} \frac{u(y)^p \,dy}{|x-y|^{n-\alpha}} \leq 2C_3 \delta^p
\end{equation*}
as in the previous proof.  Since $\log(1+R) \leq u(r)/\delta\leq \log(1+2R)$, this gives
\begin{equation*}
(I_\alpha \ast u^p) \cdot u^q \leq 2C_3 \delta^p \cdot u^q \leq C_4 \delta^{p+q}
\end{equation*}
and the result follows as before because $p+q>m-1$ by assumption.

\medskip

\noindent{\bf Proof of Theorem \ref{thm7}.}
We proceed as in the proof of Theorem \ref{thm5}, but we take $u(r)= L(1+R-r)$ for some large
enough $L>0$. Using the same approach as before, we find that
\begin{align*}
-\text{div} ( A(|\nabla u|) \nabla u) = \frac{N-1}{r} \cdot L A(L) \geq C_1 L^{m-1}
\end{align*}
for all $r<R$.  On the other hand, the convolution term satisfies the estimate
\begin{equation*}
(I_\alpha \ast u^p) \cdot u^q \leq C_4 L^{p+q}
\end{equation*}
as before.  Since $p+q<m-1$, this yields a solution for all large enough $L>0$.

\medskip

\noindent{\bf Proof of Theorem \ref{thm8}.}
We proceed as in the proof of Theorem \ref{thm6}, but we take $u(r)= L\log(1+R+r)$ for some large
enough $L>0$.  Since the argument is very similar, we omit the details.

\medskip

\section{Proofs of Theorems \ref{thm9} - \ref{thm11}}

This section is devoted to the proofs of Theorems \ref{thm9}, \ref{thm10} and \ref{thm11}.

\noindent {\bf Proof of Theorem \ref{thm9}.}
(i) Assume $(u, v)$ is a positive solution of \eqref{sys1}. Let $0\leq \theta_1,\theta_2<1$ be such
that by letting $a= q+r\theta_1$, $b= s+p\theta_2$ we have
\begin{equation}\label{s7}
r> m_1-1-a\geq 0 \;\mbox{ and } \; p> m_2-1-b\geq 0.
\end{equation}
Note that if $m_1-1=q$ (resp. $m_2-1=s$) we take $\theta_1= 0$ (resp. $\theta_2=0$).

Take $R> 2$ and $\varphi \in C_{c}^{\infty}(\R^N)$, $0\leq \varphi \leq 1$ such that
\begin{equation}\label{RR}
{\rm supp\,}\varphi \subset \Omega_R:= B_{4R}\setminus B_{R/2}\,,
\;\varphi \equiv 1\mbox{ in }B_{2R}\setminus B_{R}.
\end{equation}
By Proposition \ref{p1} we have
\begin{equation}\label{s8}
\int_{\Omega_R}(I_\alpha*v^p)u^{q-a}\varphi\leq CR^{N-m_1-\frac{m_1-1-a}{l_1}N}
\Big(\int_{\Omega_R}u^{\ell_1}\Big)^{\frac{m_1-1-a}{\ell_1}} \quad\mbox{ for } \ell_1> m_1-1-a
\end{equation}
and
\begin{equation}\label{s9}
\int_{\Omega_R}(I_\beta*u^r)v^{s-b}\varphi\leq CR^{N-m_2-\frac{m_2-1-b}{\ell_2}N}
\Big(\int_{\Omega_R}v^{\ell_2}\Big)^{\frac{m_2-1-b}{\ell_2}} \quad\mbox{ for } \ell_2> m_2-1-b.
\end{equation}
Also, for $x\in \Omega_R$ we estimate
\begin{equation}\label{s10}
(I_\alpha*v^p)(x)\geq CR^{\alpha-N}\int_{B_{4R}\setminus B_{1}}v^{p}(y) dy,
\end{equation}
and
\begin{equation}\label{s11}
(I_\beta*u^r)(x)\geq CR^{\beta-N}\int_{B_{4R}\setminus B_{1}}u^{r}(y) dy.
\end{equation}
From \eqref{s8} and \eqref{s10} we find
\begin{equation}\label{s12}
\Big(\int_{B_{4R}\setminus B_1}v^p\Big)\Big(\int_{\Omega_R}u^{q-a}\varphi \Big)
\leq CR^{2N-m_1-\alpha-\frac{m_1-1-a}{\ell_1}N} \Big(\int_{\Omega_R}u^{\ell_1}\Big)^{\frac{m_1-1-a}{\ell_1}},
\end{equation}
and similarly, from \eqref{s9} and \eqref{s11} we get
\begin{equation}\label{s13}
\Big(\int_{B_{4R}\setminus B_1}u^r\Big)\Big(\int_{\Omega_R}v^{s-b}\varphi \Big)
\leq CR^{2N-m_2-\beta-\frac{m_2-1-b}{\ell_2}N} \Big(\int_{\Omega_R}v^{\ell_2}\Big)^{\frac{m_2-1-b}{\ell_2}}.
\end{equation}
Now, take $\ell_1= r$, $\ell_2= p$ in \eqref{s12} and \eqref{s13} respectively and observe that
$\ell_1> m_1-1-a$ and $\ell_2> m_2-1-b$ due to \eqref{s7}. We next multiply \eqref{s12} and
\eqref{s13} to obtain
\begin{equation}\label{s14}
\begin{aligned}
\Big(\int_{B_{4R}\setminus B_1}u^r\Big)\Big(\int_{\Omega_R}u^{-r\theta_1}\varphi \Big)
\Big(\int_{B_{4R}\setminus B_1}v^p\Big)&\Big(\int_{\Omega_R}v^{-p\theta_2}\varphi \Big)\\
&\leq CR^{\gamma}\Big(\int_{\Omega_R}u^{r}\Big)^{\frac{m_1-1-q}{r}-\theta_1}
\Big(\int_{\Omega_R}v^{p}\Big)^{\frac{m_2-1-s}{p}-\theta_2},
\end{aligned}
\end{equation}
where
\begin{equation*}
\gamma= 4N-(m_1+m_2+\alpha+\beta)-N\Big(\frac{m_1-1-q}{r}-\theta_1\Big)-N\Big(\frac{m_2-1-s}{p}-\theta_2\Big).
\end{equation*}
By H\"{o}lder's inequality we have
\begin{equation}\label{s16}
\Big(\int_{\Omega_R}u^r\Big)^{\theta_1}\Big(\int_{\Omega_R}u^{-r\theta_1}\varphi \Big)\geq
\Big(\int_{\Omega_R}\varphi^{\frac{1}{\theta_1+1}} \Big)^{\theta_1+1}\geq CR^{N(\theta_1+1)},
\end{equation}
and
\begin{equation}\label{s17}
\Big(\int_{\Omega_R}v^p\Big)^{\theta_2}\Big(\int_{\Omega_R}v^{-p\theta_2}\varphi \Big)\geq
\Big(\int_{\Omega_R}\varphi^{\frac{1}{\theta_2+1}} \Big)^{\theta_2+1}\geq CR^{N(\theta_2+1)}.
\end{equation}
Using  \eqref{s16}-\eqref{s17} in \eqref{s14} we find
\begin{equation}\label{s18}
\Big(\int_{B_{4R}\setminus B_1}u^r\Big)^{1-\theta_1}\Big(\int_{B_{4R}\setminus B_1}v^p\Big)^{1-\theta_2}
\leq CR^{\sigma}\Big(\int_{\Omega_R}u^{r}\Big)^{\frac{m_1-1-q}{r}-\theta_1}
\Big(\int_{\Omega_R}v^{p}\Big)^{\frac{m_2-1-s}{p}-\theta_2},
\end{equation}
where
\begin{equation}\label{s19}
\sigma= 2N-(m_1+m_2+\alpha+\beta)-N\Big(\frac{m_1-1-q}{r}\Big)-N\Big(\frac{m_2-1-s}{p}\Big).
\end{equation}
From \eqref{s18} we deduce
\begin{equation}\label{s20}
\Big(\int_{B_{4R}\setminus B_1}u^{r}\Big)^{1-\frac{m_1-1-q}{r}}
\Big(\int_{B_{4R}\setminus B_1}v^{p}\Big)^{1-\frac{m_2-1-s}{p}}\leq CR^{\sigma} \quad\mbox{ for all } R> 2.
\end{equation}

\medskip

{\bf Case 1:} $\sigma< 0$. From \eqref{s2} we have
\begin{equation*}
p\geq m_2-1-s\geq 0 \mbox{ and } r\geq m_1-1-q\geq 0.
\end{equation*}
If $r> m_1-1-q$, then passing to the limit with $R\rightarrow \infty$ in \eqref{s20} we obtain
\begin{equation*}
\left\{
\begin{aligned}
\Big(\int_{\R^{N}\setminus B_1}u^{r}\Big)^{1-\frac{m_1-1-q}{r}}
\Big(\int_{\R^N\setminus B_1}v^{p}\Big)^{1-\frac{m_2-1-s}{p}}&= 0 &&\quad\mbox{ if } p> m_2-1-s, \\
\int_{\R^{N}\setminus B_1}u^{r}&= 0  &&\quad\mbox{ if } p= m_2-1-s.
\end{aligned}
\right.
\end{equation*}
This clearly contradicts the fact that $u, v$ are positive.

\medskip

Similarly, if $p> m_2-1-s$ we find $\int_{\R^{N}\setminus B_1}v^{p}= 0$  which is a contradiction
since $v$ is positive.

\medskip

Hence, $r= m_1-1-q$ and $p= m_2-1-s$. But now \eqref{s20} yields $1\leq CR^{\sigma}$ for all $R> 2$
which is again impossible if we let $R\rightarrow \infty$ (since $\sigma< 0$).

\medskip

{\bf Case 2:} $\sigma= 0$. From the definition of $\sigma$ in \eqref{s19} we deduce that the
equalities $r= m_1-1-q$ and $p= m_2-1-s$ cannot hold simultaneously. Assume for instance that $r>
m_1-1-q$. Passing to the limit with $R\rightarrow \infty$ in \eqref{s20} we find
\begin{equation*}
\left\{
\begin{aligned}
\int_{\R^{N}\setminus B_1}u^{r}&< \infty &&\quad\mbox{ if } p= m_2-1-s, \\
\int_{\R^{N}\setminus B_1}u^{r}&< \infty, \;\; \int_{\R^{N}\setminus B_1}v^{p}< \infty
&&\quad\mbox{ if } p> m_2-1-s.
\end{aligned}
\right.
\end{equation*}
In particular, as $R\rightarrow \infty$ we have
\begin{equation*}
\left\{
\begin{aligned}
\int_{\Omega_R}u^{r}&\rightarrow 0 &&\quad\mbox{ if } p= m_2-1-s, \\
\int_{\Omega_R}u^{r}&\rightarrow 0, \;\; \int_{\Omega_R}v^{p}\rightarrow 0  &&\quad\mbox{ if } p> m_2-1-s.
\end{aligned}
\right.
\end{equation*}
Using this fact in \eqref{s18} and letting $R\rightarrow \infty$ we deduce $\int_{\R^{N}\setminus
B_1}u^{r}=0$ which is again a contradiction.

\medskip

(ii) Assume $(u, v)$ is a positive solution of \eqref{sys2}. Similar to the above, let $R> 2$ and
$\varphi \in C_{c}^{\infty}(\R^N)$, $0\leq \varphi \leq 1$ that satisfies \eqref{RR}. By
Proposition \ref{p1} we have
\begin{equation}\label{s21}
\int_{\Omega_R}(I_\alpha*v^p)v^{q}u^{-a}\varphi\leq CR^{N-m_1-\frac{m_1-1-a}{\ell_1}N}
\Big(\int_{\Omega_R}u^{\ell_1}\Big)^{\frac{m_1-1-a}{\ell_1}} \quad\mbox{ for } \ell_1> m_1-1-a
\end{equation}
and
\begin{equation}\label{s22}
\int_{\Omega_R}(I_\beta*u^r)u^{s}v^{-b}\varphi\leq CR^{N-m_2-\frac{m_2-1-b}{\ell_2}N}
\Big(\int_{\Omega_R}v^{\ell_2}\Big)^{\frac{m_2-1-b}{\ell_2}} \quad\mbox{ for } \ell_2> m_2-1-b.
\end{equation}
Next, \eqref{s21} and \eqref{s22} yield
\begin{equation}\label{s23}
\Big(\int_{B_{4R}\setminus B_1}v^p\Big)\Big(\int_{\Omega_R}v^{q}u^{-a}\varphi \Big)\leq
CR^{2N-m_1-\alpha-\frac{m_1-1-a}{\ell_1}N} \Big(\int_{\Omega_R}u^{\ell_1}\Big)^{\frac{m_1-1-a}{\ell_1}},
\end{equation}
and
\begin{equation}\label{s24}
\Big(\int_{B_{4R}\setminus B_1}u^r\Big)\Big(\int_{\Omega_R}u^{s}v^{-b}\varphi \Big)\leq
CR^{2N-m_2-\beta-\frac{m_2-1-b}{\ell_2}N} \Big(\int_{\Omega_R}v^{\ell_2}\Big)^{\frac{m_2-1-b}{\ell_2}}.
\end{equation}
From \eqref{s4} we can find $0\leq \tau_1,\tau_2<1$ such that, letting $b= q+p\tau_2$, $a=
s+r\tau_1$ we have
\begin{equation}\label{s25}
p> m_2-1-a\geq 0 \quad\mbox{ and } r> m_1-1-b\geq 0.
\end{equation}
If $m_2-1= q$ then we take $\tau_2= 0$ (and similarly if $m_1-1= s$ we take $\tau_1= 0$). In light
of \eqref{s25}, we may take $\ell_1= r$, $\ell_2= p$ and $b= q+p\tau_2$, $a= s+r\tau_1$ in
\eqref{s23} and \eqref{s24}. Multiplying \eqref{s23} and \eqref{s24} we obtain
\begin{equation}\label{s26}
\begin{aligned}
\Big(\int_{B_{4R}\setminus B_1}u^r\Big)\Big(\int_{B_{4R}\setminus B_1}v^p\Big)
\Big(\int_{\Omega_R}u^{s}v^{-b}\varphi \Big)&\Big(\int_{\Omega_R}v^{q}u^{-a}\varphi \Big)\\
&\leq CR^{\gamma}\Big(\int_{\Omega_R}u^{r}\Big)^{\frac{m_1-1-s}{r}-\tau_1}
\Big(\int_{\Omega_R}v^{p}\Big)^{\frac{m_2-1-q}{p}-\tau_2},
\end{aligned}
\end{equation}
where
\begin{equation*}
\gamma= 4N-(m_1+m_2+\alpha+\beta)-N\Big(\frac{m_1-1-s}{r}-\tau_1\Big)-N\Big(\frac{m_2-1-q}{p}-\tau_2\Big).
\end{equation*}
Using H\"{o}lder's inequality we have
\begin{equation}\label{s28}
\begin{aligned}
\Big(\int_{\Omega_R}u^r\Big)^{\tau_1}\Big(\int_{\Omega_R}v^p\Big)^{\tau_2}
\Big(\int_{\Omega_R}u^{s}v^{-b}\varphi \Big)\Big(\int_{\Omega_R}v^{q}u^{-a}\varphi \Big)
&\geq \Big(\int_{\Omega_R}u^{\frac{r\tau_1+s-a}{2+\tau_1+\tau_2}}v^{\frac{p\tau_2+q-b}{2+\tau_1+\tau_2}}
\varphi^{\frac{2}{2+\tau_1+\tau_2}}\Big)^{2+\tau_1+\tau_2}\\
&= \Big(\int_{\Omega_R}\varphi^{\frac{2}{2+\tau_1+\tau_2}}\Big)^{2+\tau_1+\tau_2}\\
&= CR^{N(2+\tau_1+\tau_2)}.
\end{aligned}
\end{equation}
Using \eqref{s28} in \eqref{s26} we find
\begin{equation}\label{s29}
\Big(\int_{B_{4R}\setminus B_1}u^r\Big)^{1-\tau_1}\Big(\int_{B_{4R}\setminus B_1}v^p\Big)^{1-\tau_2}
\leq CR^{\tau}\Big(\int_{\Omega_R}u^{r}\Big)^{\frac{m_1-1-s}{r}-\tau_1}
\Big(\int_{\Omega_R}v^{p}\Big)^{\frac{m_2-1-q}{p}-\tau_2},
\end{equation}
where
\begin{equation*}
\tau= 2N-(m_1+m_2+\alpha+\beta)-N\Big(\frac{m_1-1-s}{r}+\frac{m_2-1-q}{p}\Big)\leq 0.
\end{equation*}
Next, we arrive again at a contradiction by considering the cases $\tau> 0$ and $\tau= 0$ using a
similar argument as in part (i).

\medskip

(iii) This follows with the same arguments as in part (ii) above.

\qed

\medskip

\noindent{\bf Proof of Theorem \ref{thm10}}
(i) Using Lemma \ref{lbas}, system \eqref{sys1} reduces to the following decoupled inequalities 
$$
\left\{
\begin{aligned}
L_\mathcal{A} u\geq |x|^{\alpha-N}u^q \\
L_\mathcal{B} v\geq |x|^{\beta-N}v^s 
\end{aligned}
\right.
\qquad\mbox{ in } \R^N\setminus B_{2}.
$$
The conclusion follows from Proposition \ref{p3}.

\medskip

\noindent (ii) Using Lemma \ref{lbas}, systems \eqref{sys2}  and \eqref{sys3} reduce to
$$ \left\{
\begin{aligned}
L_\mathcal{A} u\geq |x|^{\alpha-N}v^q \\
L_\mathcal{B} v\geq |x|^{\beta-N}u^s 
\end{aligned}
\right.
\qquad\mbox{ in } \R^N\setminus B_{2}.
$$
The conclusion follows from Proposition \ref{psys}. \qed

\medskip

\noindent{\bf Proof of Theorem \ref{thm11}.}
We shall only establish part (i), as the other two parts are similar.  We proceed as in the proof
of Theorem \ref{thm4} and we look for solutions of the form
\begin{align*}
u(r) &= \e (1+r)^{-\gamma_1} \cdot \left( 1 - \frac{k}{1+ \log (1+r)} \right), \\
v(r) &= \e (1+r)^{-\gamma_2} \cdot \left( 1 - \frac{k}{1+ \log (1+r)} \right)
\end{align*}
with $\gamma_1= \frac{N-m_1}{m_1-1}$, $\gamma_2= \frac{N-m_2}{m_2-1}$ and $\e,k>0$ sufficiently
small.  Using the exact same approach as before, one may establish the estimate
\begin{align}
L_{\mathcal{A}}u
\geq \frac{C_4 \e^{m-1} (1+r)^{-\gamma_1(m_1-1) -m_1} }{(1 + \log (1+r))^{2m_1}} =
\frac{C_4 \e^{m_1-1} (1+r)^{-N}}{(1 + \log (1+r))^{2m_1}}
\end{align}
in analogy with \eqref{est5}.  Since $\gamma_2 p = \frac{p(N-m_2)}{m_2-1} > \alpha$, one also has
\begin{equation*}
(I_\alpha \ast v^p) u^q \leq
\left\{ \begin{array}{lc}
C_5 \e^{p+q} (1+r)^{\alpha - \gamma_1 q - \gamma_2 p} &
\text{if \;}\, \gamma_2 p<N, \\
C_5 \e^{p+q} (1+r)^{\alpha -\gamma_1 q -N} (1+\log (1+r)) &\text{if \;}\,\gamma_2 p=N, \\
C_5 \e^{p+q} (1+r)^{\alpha -\gamma_1 q -N} & \text{if \;}\, \gamma_2 p>N,
\end{array} \right.
\end{equation*}
by Lemma \ref{lbas}(iii) in analogy with \eqref{est6}.  Now, the given assumptions assert that
\begin{equation*}
\gamma_1 q > \alpha, \qquad \gamma_1 q + \gamma_2 p > \alpha + N, \qquad p+q > m_1 -1.
\end{equation*}

\noindent \textbf{Case 1:} $\gamma_2 p < N$. Then one may conclude that
\begin{align*}
L_{\mathcal{A}}u
&\geq \frac{C_4 \e^{m_1-1} (1+r)^{-N}}{(1 + \log (1+r))^{2m_1}} \\
&\geq C_5 \e^{p+q} (1+r)^{\alpha - \gamma_1 q - \gamma_2 p} \geq (I_\alpha \ast v^p) u^q
\end{align*}
for all sufficiently small $\e>0$ because $\gamma_1 q + \gamma_2 p> \alpha +N$ and $p+q>m_1-1$.

\noindent \textbf{Case 2:} $\gamma_2 p \geq N$. Then one may similarly conclude that
\begin{align*}
L_{\mathcal{A}}u
&\geq \frac{C_4 \e^{m_1-1} (1+r)^{-N}}{(1 + \log (1+r))^{2m_1}} \\
&\geq C_5 \e^{p+q} (1+r)^{\alpha -\gamma_1 q -N} \cdot (1+\log (1+r))
\geq (I_\alpha \ast v^p) u^q
\end{align*}
for all sufficiently small $\e>0$ because $\gamma_1 q > \alpha$ and $p+q>m_1-1$.

In either case then, the first equation of the system \eqref{sys1} is satisfied for all small
enough $\e>0$. Using the exact same approach, one finds that the second equation holds as well.
\qed

\end{document}